\documentclass[12pt,a4paper,reqno]{amsart}
\usepackage{geometry}
\usepackage{hyperref}
\usepackage{graphicx}   
\usepackage{amsmath,amsthm,amssymb}
\usepackage{xcolor}
\usepackage{mathtools}
\usepackage{yfonts}
\usepackage[T1]{fontenc}

\newtheorem{corollary}{Corollary}
\newtheorem{theorem}{Theorem}

\newtheorem{lemma}{Lemma}
\theoremstyle{remark}
\newtheorem*{remark}{Remark}

\numberwithin{equation}{section}

\renewcommand{\Re}{\mathfrak{Re}}

\begin{document}

\title{Moments of the Riemann zeta function at its local extrema}

\author{Andrew Pearce-Crump}
\address{School of Mathematics, Fry Building, Woodland Road, Bristol, BS8 1UG, United Kingdom}
\email{andrew.pearce-crump@bristol.ac.uk}

\begin{abstract}
Conrey, Ghosh and Gonek studied the first moment of the derivative of the Riemann zeta function evaluated at the non-trivial zeros of the zeta function, resolving a problem known as Shanks' conjecture. Conrey and Ghosh studied the second moment of the Riemann zeta function evaluated at its local extrema along the critical line to leading order. 

In this paper we combine the two results, evaluating the first moment of the zeta function and its derivatives at the local extrema of zeta along the critical line, giving a full asymptotic. We also consider the factor from the functional equation for the zeta function at these extrema.
\end{abstract}

\maketitle

\section{Background}\label{sect:Background}
Shanks' Conjecture \cite{ShanksConj} is that 
\[
\zeta '(\rho) \text{ is real and positive in the mean}
\]
as $\rho = \beta + i \gamma$ ranges over non-trivial zeros of the Riemann zeta function $\zeta (s)$. This was proved by Conrey, Ghosh and Gonek \cite{CGG88}, who showed 
\begin{equation}
\sum_{0 < \gamma \leq T} \zeta ' (\rho) = \frac{T}{4 \pi} (\log T)^2 + O (T \log T) \label{1stDeriv}
\end{equation}
as $T \rightarrow \infty$ as a side calculation while giving a proof that there are infinitely many simple zeros of the Riemann zeta function. Fujii \cite{Fuj94,FujiiDist} found a full asymptotic for \eqref{1stDeriv}, with a power-saving error term under the Riemann Hypothesis.

The Generalised Shanks' Conjecture can be stated that on average,
\[
\zeta^{(n)}(\rho) \text{ is real and positive if $n$ is odd, negative if $n$ is even.}
\]
Kaptan, Karabulut and Y{\i}ld{\i}r{\i}m \cite{KKY11}, using similar ideas to Conrey, Ghosh and Gonek, found that
\begin{equation}\label{nthderiv}
\sum_{0 < \gamma \leq T} \zeta^{(n)} (\rho) =
  \frac{(-1)^{n+1}}{n+1} \frac{T}{2 \pi} (\log T)^{n+1} + O \bigl(T (\log T)^n \bigr)
\end{equation}
as $T \rightarrow \infty$ which clearly implies and confirms the previously unnoticed Generalised Shanks' Conjecture. Hughes and Pearce-Crump \cite{HugPC22} found a full asymptotic for \eqref{nthderiv}, with a power-saving error term under the Riemann Hypothesis. We state this result in Theorem~\ref{thm:HugPC} in Section~\ref{sect:PrelimLem}. For a full history of this problem see \cite{HugMarPea24}.

Under the Riemann Hypothesis, Conrey and Ghosh \cite{ConGho85} showed that
    \begin{equation}\label{CGLead}
       \sum_{0< \gamma \leq T} \max_{\gamma \leq t \leq \gamma^+} \left| \zeta \left( \tfrac{1}{2} + it \right) \right|^2 \sim \frac{e^2 - 5}{4 \pi} T (\log T)^2
    \end{equation}
as $T\rightarrow \infty$ where $\gamma \leq \gamma^+$ are successive ordinates of non-trivial zeros of $\zeta (s)$.

In \cite{HugLugPea24}, Hughes, Lugmayer and Pearce-Crump calculated the full asymptotic expansion of \eqref{CGLead}, showing that under the Riemann Hypothesis, for $K$ a positive integer, that 
\begin{multline}\label{HLPC}
   \sum_{0< \gamma \leq T} \max_{\gamma \leq t \leq \gamma^+} \left| \zeta \left( \tfrac{1}{2} + it \right) \right|^2  =  \frac{e^2-5}{2} \frac{T}{2\pi} \left( \log \frac{T}{2\pi} \right)^2 +  \left(5-e^2-10 \gamma_0 +2 e^2 \gamma_0 \right) \frac{T}{2\pi}\left( \log \frac{T}{2\pi} \right) \\
+ \frac{T}{2 \pi} \sum_{k=0}^{K} \frac{c_k}{(\log T/2\pi)^k} + O_K\left(\frac{T}{(\log T)^{K+1}}\right)
\end{multline}
where $\gamma_0$ is Euler's constant and the other constants $c_k$ are effectively computable constants.

Recall that the Hardy $Z$-function is defined by 
\[
Z(t) = e^{i \theta (t)} \zeta \left( \tfrac{1}{2} + it \right),
\]
where $\theta (t)$ is the Riemann–Siegel theta function given by
\[
e^{-2 i \theta (t)} = \chi \left( \tfrac{1}{2} + it \right) 
\]
and where $\chi (s)$ is the factor from the functional equation of $\zeta (s)$, given by 
\begin{equation}\label{eq:FE}
\zeta (s) = \chi (s) \zeta (1-s).
\end{equation}
Note that $Z(t)$ is real for real $t$ and satisfies $|Z(t)| = |\zeta (1/2 +it)|$ which means we can rewrite \eqref{CGLead} and \eqref {HLPC} in terms of the Hardy Z-function.

For $\gamma \leq \gamma^+$ consecutive positive ordinates of non-trivial zeros of $\zeta (s)$, we know there is exactly one zero of $Z'(t)$ (under the Riemann Hypothesis), call it $\lambda$, with $\gamma \leq \lambda \leq \gamma^+$. A proof of this result can be found in \cite{IvicBook}. This suggest an unambiguous way of writing the following results, where we are summing over the relative extrema of $|\zeta (s)|$ on the critical line. We will write our sums over these $\lambda$, where $0 < \lambda \leq T$ and $Z'(\lambda)=0$. 

Note that there are two zeros of $Z'(t)$ between $t=0$ and the first positive zero of $Z(t)$. As these additional zeros only introduce an error of $O(1)$, we will ignore them as this contribution is within our larger error terms in the results. 

Specifically, under the Riemann Hypothesis, we can rewrite the result of Conrey and Ghosh \cite{ConGho85} as
    \begin{equation}\label{CGLeadZ}
       \sum_{0< \lambda \leq T} Z(\lambda)^2 \sim \frac{e^2 - 5}{4 \pi} T (\log T)^2
    \end{equation}
as $T\rightarrow \infty$, where we recall that the points $\lambda$ that are being summed over are the zeros of the derivative of the Hardy $Z$-function up to a height $T$.

Finally, we remind the reader that results such as \eqref{CGLead} and \eqref{HLPC} are linked to the joint moments of the Riemann zeta function and the joint moments of the Hardy $Z$-function. To see why this is true, see \cite{ConGho89, Milinovich11}. Specifically, Conrey and Ghosh showed that    
    \begin{equation*}
        \int_{1}^{T} |Z(t)||Z'(t)| \ dt \sim \frac{e^2 - 5}{4 \pi} T (\log T)^2
    \end{equation*} 
as $T \rightarrow \infty$, and so the lower-order terms follows from \eqref{HLPC}. For a more detailed history of this rich problem, we refer the reader to \cite{HugLugPea24} and the references therein. 

\section{Results}\label{sect:Results}
In this paper we combine some of these results, specifically equations \eqref{CGLead} and \eqref{nthderiv}, or rather the full asymptotic expansions given in \eqref{HLPC} and Theorem \ref{thm:HugPC}. We also compare our results with those given in Section~\ref{sect:Background}. In these comparisons we don't give full reasons beyond the obvious (that we are summing over different points/summing different functions), but the differences are worth pointing out explicitly in the remarks that follow the results and we have endeavoured to give some heuristic reasoning. 

The Stieltjes constants will appear in several of the following results. The Laurent expansion for $\zeta (s)$ about $s=1$ can be written in terms of these constants $\gamma_j$, given by
\begin{equation}\label{eq:LaurentZeta}
\zeta(s) = \frac{1}{s-1} + \gamma_0 - \gamma_1 (s-1) + \frac{1}{2!} \gamma_2 (s-1)^2 + \dots + (-1)^j \frac{1}{j!} \gamma_j (s-1)^j + \dots.
\end{equation}

We begin by evaluating the derivatives of the Riemann zeta function over the local extrema of $|\zeta (s)|$. We split the cases of derivatives and no derivatives into two separate results, both for clarity of the results and as the proofs are slightly different. The no derivative case is non-trivial, compared with typical moments where the sums are over non-trivial zeros of the Riemann zeta function.

\begin{theorem}\label{thm:FirstMomnthDerivRho1}
    Assume the Riemann Hypothesis. Write $L~=~\log T/2\pi$. For $\lambda$ defined by $Z'(\lambda)=0$ and $K, n \geq 1$ positive integers, we have 
\begin{equation*}
    \sum_{0< \lambda \leq T} \zeta^{(n)} \left( \frac{1}{2} + i\lambda \right) = a_{n+1} \frac{T}{2\pi}L^{n+1} + \frac{T}{2\pi} \sum_{\ell=0}^{n} a_{n-\ell} L^{n-\ell} + \frac{T}{2\pi} \sum_{m=1}^{K} \frac{b_m}{L^{m}} + O\left(\frac{T}{L^{K+1}} \right),
\end{equation*}
as $T \rightarrow \infty$ where:
\begin{enumerate}
    \item The leading order coefficient is given by
    \[
    a_{n+1} = (-1)^{n} \left( \frac{e^2-2}{n+1} + (-1)^{n+1} \frac{n!}{2^{n+1}} \left(1 - e^2 \sum_{k=0}^{n+1} \frac{(-2)^k}{k!} \right) \right).
    \]
    \item The subleading, non-negative logarithm power coefficients are given by    
    \begin{align*}
        a_{n-\ell} = (-1)^{n}& \sum_{k=1}^{\infty} 2^k \sum_{j=0}^{n} \binom{n}{j} \left[ \frac{c_{\ell+1}^{k,j}}{(k+j-\ell)!} + (k-n+j) \frac{c_{\ell}^{k,j}}{(k+j+1-\ell)!} \right. \\
        &+ \left. (k-n+j) \sum_{m=0}^{\ell-1} (-1)^{\ell -m} (n-m) \dots (n-\ell+1) \frac{c_{m}^{k,j}}{(k+j+1-m)!}  \right] \\
        +(-1)^{n+1} &\sum_{\ell=0}^n \binom{n}{\ell}  (-1)^ {\ell}  \ell!  \left(-1 + \sum_{m=0}^{\ell} \frac{1}{m!}\gamma_m \right),
    \end{align*}
where if $\ell=0$ the summation over $m$ in the square brackets is empty, and where if $\ell=n$, the coefficient of $L^0$, that is to say $a_0$, has an extra contribution of $n! A_n$, where the $A_n$ are the coefficients in the Laurent expansion of $\zeta' (s)/\zeta (s)$ about $s=1$.
    \item The subleading, negative logarithm power coefficients are given by
    \begin{align*}
    b_{m} = (-1)^{n}& \sum_{k=1}^{\infty} 2^k \sum_{j=0}^{n} \binom{n}{j} \beta_{m}^{k,j},
    \end{align*}
    where for $1 \leq m \leq k+j-n$, we have
    \[
    \beta_{m}^{k,j} = \frac{c_{m+n+1}^{k,j}}{(k+j-n-m)!} + \frac{(m-1)!}{k-n+j-1} \sum_{\ell=0}^{m-1}  \binom{k-n+j}{\ell} c_{\ell+n+1}^{k,j}
    \] 
    and where for $m \leq k+j-n+1$, we have
    \[
    \beta_{m}^{k,j} = \frac{(m-1)!}{k-n+j-1} \sum_{\ell=0}^{k+j-n}  \binom{k-n+j}{\ell} c_{\ell+n+1}^{k,j}.
    \]
\end{enumerate}
In these coefficients $a_m, b_m$, the $c_{\ell}^{k,j}$ are the Laurent series coefficients around $s=1$ of
    \[
    \left(\frac{\zeta '}{\zeta} (s) \right)' \left(-\frac{\zeta '}{\zeta} (s) \right)^{k-1} \zeta^{(j)} (s) \frac{1}{s} = \sum_{\ell=0}^\infty c_{\ell}^{k,j} (s-1)^{-k-j-2+\ell}.
    \]
\end{theorem}

An immediate corollary of Theorem \ref{thm:FirstMomnthDerivRho1} follows by setting $n=1$. We state this result explicitly to compare it with \eqref{1stDeriv} and Shanks' Conjecture. 

\begin{corollary}\label{cor:FirstMom1stDerivRho1}
    Assume the Riemann Hypothesis. Write $L~=~\log T/2\pi$. For $\lambda$ defined by $Z'(\lambda)=0$ and $K \geq 1$ a positive integer, we have 
\begin{align*}
   \sum_{0< \lambda \leq T} \zeta' \left( \frac{1}{2} + i\lambda \right)  =  &-\frac{(e^2-3)}{4} \frac{T}{2 \pi} L^2 + \frac{e^2-3+2\gamma_0}{2} \frac{T}{2 \pi} L + \frac{3-e^2(1+2\gamma_0+2 \gamma_1) }{2}\frac{T}{2 \pi} \\
   &+\frac{T}{2 \pi} \sum_{m=1}^{K} \frac{c_m}{L^k} + O_K\left(\frac{T}{L^{K+1}}\right)
\end{align*}
as $T \rightarrow \infty$, where the other constants $c_m$ are special cases of the coefficients $b_m$ from Theorem \ref{thm:FirstMomnthDerivRho1} with $n=1$.
\end{corollary}

We may also prove the analogous result of Conrey and Ghosh in \eqref{CGLead} from \cite{ConGho85}, with the modulus squared removed.

\begin{theorem}\label{thm:FirstMomRho1}
    Assume the Riemann Hypothesis. Write $L~=~\log T/2\pi$. For $\lambda$ defined by $Z'(\lambda)=0$ and $K \geq 1$ a positive integer, we have 
\begin{equation*}
    \sum_{0< \lambda \leq T} \zeta \left( \frac{1}{2} + i\lambda \right)  = \frac{e^2-3}{2} \frac{T}{2 \pi} L + \frac{3-e^2-4\gamma_0}{2} \frac{T}{2 \pi} + \frac{T}{2 \pi} \sum_{m=1}^{K} \frac{d_m}{L^m} + O_K\left(\frac{T}{L^{K+1}}\right)
\end{equation*}
as $T \rightarrow \infty$, where 
\[
d_m = \sum_{k=m}^{\infty} 2^k \frac{c_{k,m+1}}{(k-m)!} + (m-1)! \sum_{k=1}^{\infty} \frac{ 2^k k}{(k-1)!} \sum_{\ell=0}^{\min\{m-1, k\}} \binom{k}{\ell} c_{k,\ell+1} ,
\]
and where the $c_{k,\ell}$ are the Laurent series coefficients around $s=1$ of
    \[
    \left(\frac{\zeta '}{\zeta} (s) \right)' \left(-\frac{\zeta '}{\zeta} (s) \right)^{k-1} \zeta (s) \frac{1}{s} = \sum_{\ell=0}^\infty c_{k, \ell} (s-1)^{-k-2+\ell} .
    \]
\end{theorem}

We evaluate the sum of the factor $\chi (s)$ from the functional equation \eqref{eq:FE} for $\zeta (s)$ over the local extrema of $\zeta (s)$. 
\begin{theorem}\label{thm:ChiRho1}
    Assume the Riemann Hypothesis. Write $L~=~\log T/2\pi$. For $\lambda$ defined by $Z'(\lambda)=0$ and $K \geq 1$ a positive integer, we have 
\begin{equation*}
    \sum_{0< \lambda \leq T} \chi \left( \frac{1}{2} + i\lambda \right)  = (e^2-2) \frac{T}{2 \pi} - 4e^2\gamma_0 \frac{T}{2 \pi} \frac{1}{L} + \frac{T}{2 \pi} \sum_{m=2}^{K} \frac{e_m}{L^m} + O_K\left(\frac{T}{L^{K+1}}\right)
\end{equation*}
as $T \rightarrow \infty$, where 
    \[
    e_m = \sum_{k=m}^{\infty} 2^k \frac{c_{k,m}}{(k-m)!} + (m-1)! \sum_{k=1}^{\infty} \frac{ 2^k }{(k-1)!} \sum_{\ell=0}^{\min\{m-1, k\}} \binom{k}{\ell} c_{k,\ell} ,
    \]
and where the $c_{k,\ell}$ are the Laurent series coefficients around $s=1$ of
    \[
    \left(\frac{\zeta '}{\zeta} (s) \right)' \left(-\frac{\zeta '}{\zeta} (s) \right)^{k-1} \frac{1}{s} = \sum_{\ell=0}^\infty c_{k,\ell} (s-1)^{-k-1+\ell} .
    \]
\end{theorem}

\begin{remark}
    The infinite descending chain of powers of $\log T$ in the previous results where each function is summed over the relative extrema of zeta is due to a pole in the function that we consider near $1$, given by 
    \[
    \beta_Z =1+2 / L+O\left(L^{-2}\right),
    \]
    where $L=\log T / 2 \pi$. 
    
    We expect that we would be able to give these results with a power-saving error term of $O \left(T^{1/2+\varepsilon} \right)$ for all $\varepsilon>0$ if we were able to find a closed form for the asymptotic as a function of $\beta_Z$. 
    
    As in \cite{HugLugPea24}, we have set aside the search for such a form for the time being.
\end{remark}

With only minor adjustments to our proof of Theorem \ref{thm:ChiRho1} (which we will specify in Section \ref{sect:ProofThmChi}), we are also able to prove the following result. It is a special case of a result from \cite{KaraYil11}, but here we give it with a power-saving error term. 

\begin{theorem}\label{thm:ChiRho}
    Assume the Riemann Hypothesis. For $\rho = 1/2 + i\gamma$ a non-trivial zero of $\zeta (s)$, we have 
\begin{equation*}
   \sum_{0< \gamma \leq T} \chi \left( \frac{1}{2} + i\gamma \right)  = - \frac{T}{2 \pi} + O\left( T^{1/2 +\varepsilon} \right)
\end{equation*}
as $T \rightarrow \infty$, where $\chi (s)$ is the factor from the functional equation for $\zeta (s)$ given in \eqref{eq:FE}.
\end{theorem}

\subsection{Remarks on the previous results}
    \begin{enumerate}
        \item \textbf{Theorem \ref{thm:FirstMomnthDerivRho1}} \hfill \\
        We compare the result from this theorem with that of Kaptan, Karabulut and Y{\i}ld{\i}r{\i}m \cite{KKY11} and Hughes and Pearce-Crump \cite{HugPC22} on the Generalised Shank's conjecture. In that result, the average of $\zeta^{(n)} (s)$, summed over the non-trivial zeros of $\zeta (s)$, is clearly real and alternates between positive and negative as $n$ increases. 
        
        In the case of Theorem \ref{thm:FirstMomnthDerivRho1}, the behaviour is reversed; the average of $\zeta^{(n)} (s)$, summed over the maxima of $|\zeta (s)|$, is real and alternates between negative and positive as $n$ increases.

        It isn't immediately obvious from the leading order coefficient that this is the case, that is, that
        \[
        (-1)^n \left (\frac{e^2-2}{n+1} + (-1)^{n+1} \frac{n!}{2^{n+1}} \left(1 - e^2 \sum_{k=0}^{n+1} \frac{(-2)^k}{k!} \right)\right)
        \]
        alternates. Begin by noting by Taylor's theorem with remainder that 
        \[
        e^{-2} = \sum_{k=0}^{n+1} \frac{(-2)^k}{k!} + O \left( \frac{1}{(n+2)!} \right)
        \]
        so the coefficient of the leading order can be rewritten as
        \[
        (-1)^n \left (\frac{e^2-2}{n+1} + (-1)^{n+1} \frac{n!}{2^{n+1}} \left(1 - e^2 \left\{ e^{-2} +  O \left( \frac{1}{(n+2)!} \right) \right\} \right)\right)
        \]
        and simplifying gives
        \[
        (-1)^n \left (\frac{e^2-2}{n+1} + O \left( \frac{1}{n^2} \right) \right).
        \]
        Since $e^2-2 > 0$, the alternating behaviour is then clear, but in reverse to the Generalised Shanks' Conjecture as $n$ increases.

        \item \textbf{Corollary \ref{cor:FirstMom1stDerivRho1}} \hfill \\
        As noted in the general $n$\textsuperscript{th} derivative case, while Shanks' conjecture is that $\zeta'(\rho)$ is positive and real on average, the same question with regards to $\zeta ' (\lambda)$ is negative and real on average. 
        
        We plot the first few terms of the asymptotic against the real part of the true value of the sum in Figure \ref{fig:plotZetaPrimeLambdaAllGraphs}. If we were to plot the imaginary part of the true value, we would find that it remains small (at most $\approx \pm 10$ in absolute value).

        \begin{figure}[ht] 
        \begin{center}
        \includegraphics[height=5.8cm]{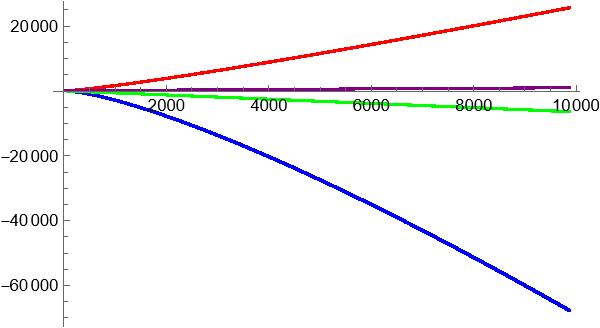}
        \end{center}
        \caption{The real part of the sum $\sum_{0 < \lambda \leq T} \zeta ' (1/2+ i\lambda)$ with true value (blue), the true - leading (red), true - leading - subleading (green) and true - leading - subleading - subsubleading (purple) terms subtracted from the true value, for $T$ up to the height of the 10,000\textsuperscript{th} zero $\lambda$ of $Z'(t)$.}\label{fig:plotZetaPrimeLambdaAllGraphs}
        \end{figure}

        \item \textbf{Theorem \ref{thm:FirstMomRho1}} \hfill \\
        The result of Conrey and Ghosh in \eqref{CGLead} has coefficient $(e^2-5)/2$ to leading order, with an extra power of a logarithm, compared with the $(e^2-3)/2$ here. The power of the logarithm can be explained by noting that we are taking a first moment in Theorem \ref{thm:FirstMomRho1}. 
        
        The numerical value of the coefficient in the first moment case is almost double that of the second moment. Curiously, the numerical value of the first subleading coefficient in \eqref{HLPC} and in Theorem \ref{thm:FirstMomRho1} are very similar in size (agreeing to the second decimal point).

        We plot the first few terms of the asymptotic against the real part of the true value of the sum in Figure \ref{fig:plotZetaLambdaAllGraphs}. If we were to plot the imaginary part of the true value, we would find that it remains small (at most $\approx \pm 20$ in absolute value).

        \begin{figure}[ht] 
        \begin{center}
        \includegraphics[height=5.8cm]{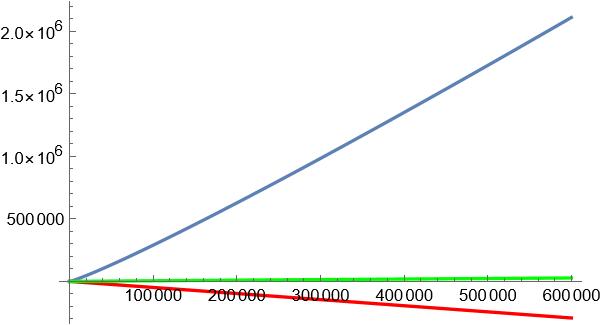}
        \end{center}
        \caption{The real part of the sum $\sum_{0 < \lambda \leq T} \zeta (1/2+ i\lambda)$ with true value (blue), the true - leading (red) and true - leading - subleading (green) terms subtracted from the true value, for $T$ up to the height of the 1,000,000\textsuperscript{th} zero $\lambda$ of $Z'(t)$.}\label{fig:plotZetaLambdaAllGraphs}
        \end{figure}

        \item \textbf{Theorem \ref{thm:ChiRho1}} \hfill \\
        We plot the first few terms of the asymptotic against the real part of the true value of the sum in Figure \ref{fig:plotChiLambdaAllGraphs}. If we were to plot the imaginary part of the true value, we would find that it remains small (at most $\approx \pm 10$ in absolute value).

        We note that the sign of the leading order term of this asymptotic is opposite that found in Theorem \ref{thm:ChiRho} and in \cite{KaraYil11}. Numerically it is also considerably larger.

        \begin{figure}[ht] 
        \begin{center}
        \includegraphics[height=5.8cm]{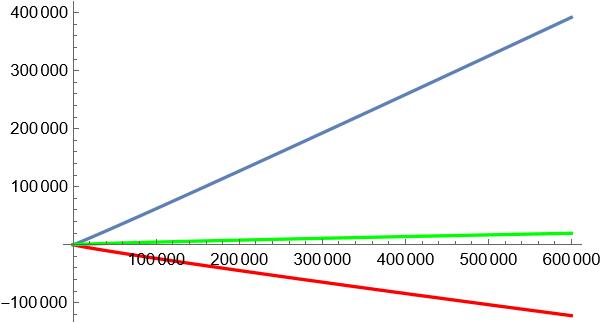}
        \end{center}
        \caption{The real part of the sum $\sum_{0 < \lambda \leq T} \chi (1/2+ i\lambda)$ with true value (blue), the true - leading (red) and true - leading - subleading (green) terms subtracted from the true value, for $T$ up to the height of the 1,000,000\textsuperscript{th} zero $\lambda$ of $Z'(t)$.}\label{fig:plotChiLambdaAllGraphs}
        \end{figure}

        \item \textbf{Theorem \ref{thm:ChiRho}} \hfill \\
        We assume the Riemann Hypothesis in proving this result but can make it unconditional by taking better care when calculating the error term.

        The behaviour is somewhat strange, however. Recall that the Gram points occur when $\theta (g_n) = n\pi$, and by the definition of the Riemann--Siegel function, we have
        \[
        \chi \left( \frac{1}{2} + ig_n \right) = e^{-2 i \theta (g_n)} = e^{-2 i n \pi} = 1,
        \]
        so 
        \[
        \sum_{0 < \gamma \leq T} \chi \left( \frac{1}{2} + ig_n \right) = \frac{T}{2 \pi} \log T + O(T). 
        \]
        Given that a zero is expected to lie roughly in the middle of a Gram interval, it isn't overly surprising that the sum in Theorem \ref{thm:ChiRho} is negative, but this doesn't account for the loss of the logarithm. Likewise, it then isn't overly surprising that the sum in Theorem \ref{thm:FirstMomRho1} over $\lambda$ is positive as we should expect these points to lie closer to the Gram points, but again doesn't account for the loss of the logarithm. 

        We plot the true value of the sum minus the main term of the asymptotic in Figure \ref{fig:plotChiZerosError}. As we can see, all that remains is clearly an error term. We have plotted both the real and imaginary parts as both are small and can be easily seen against each other, unlike in the previous figures where the imaginary parts are too small to be seen.

        \begin{figure}[ht] 
        \begin{center}
        \includegraphics[height=5.8cm]{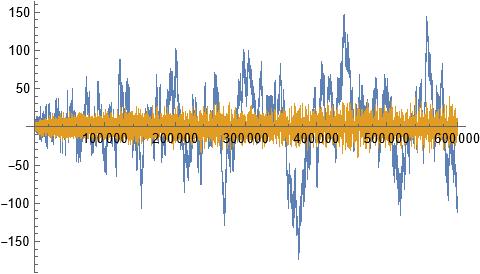}
        \end{center}
        \caption{The real (blue) and imaginary (orange) part of $\sum_{0 < \gamma \leq T} \chi(1/2+i\gamma) + \frac{T}{2 \pi}$ for $T$ up to the height of the 1,000,000\textsuperscript{th} zero $1/2+i\gamma$ of $\zeta (s)$.}\label{fig:plotChiZerosError}
        \end{figure}     
    \end{enumerate}

\subsection{Overview of the proofs}\label{sect:Overview}
In Section \ref{sect:PrelimLem}, we consider properties of the auxiliary function
\[
Z_1(s)=\zeta^{\prime}(s)-\frac{1}{2} \frac{\chi^{\prime}}{\chi}(s) \zeta(s)
\]
used in the proof, where $\chi(s)$ is the factor from the non-symmetrical form of the functional equation of the Riemann zeta function, given in \eqref{eq:FE}. This function is zero at the extrema of the Riemann zeta function. We also list other useful lemmas in this section, including functional equations, stationary phase methods, and a moments result from \cite{HugPC22} on the average of derivatives of $\zeta (s)$, evaluated over the non-trivial zeros of $\zeta (s)$.

In Section \ref{sect:ProofThmDeriv}, we prove Theorem \ref{thm:FirstMomnthDerivRho1}. This involves using Cauchy's Theorem to write the sum over the extrema as an integral  
\begin{equation*}
\sum_{0< \lambda \leq T} \zeta^{(n)} \left( \frac{1}{2} + i\lambda \right)  = \frac{1}{2\pi i} \int_{\mathcal{C}} \frac{Z_1'}{Z_1}(s) \zeta^{(n)} (s) \ ds
\end{equation*}
(up to a small error), where $\mathcal{C}$ is a contour we specify later that encloses the critical line up to a height $T$. 

We begin by showing that only the left-hand side of the contour contributes to the asymptotic, with the rest going into a power-saving error term. The bulk of the argument then comes down to evaluating
\[
\frac{1}{2 \pi i} \int_{c+i}^{c+iT} \chi(1-s) \frac{Z_1^{\prime}}{Z_1}(s) \zeta^{(n)}(1-s) \ ds
\]
for some specific $c>1$ defined in the proof.

We are able to write the logarithmic derivative of $Z_1 (s)$ as a series of the form
\[
\frac{Z_1'}{Z_1}(s) = \sum_{m=1}^\infty \frac{a(m,s)}{m^s}
\]
which converges for some $c>1$ with coefficients $a(n,s)$ given in Lemma \ref{lem:Dirichletlogderiv}. Using this and the Dirichlet series for the derivatives of $\zeta (s)$ allows us to apply the method of stationary phase to change this integral into a sum, given by
\begin{equation*}
(-1)^{n} \left(  \sum_{m_1m_2 \leq T/2 \pi} \Lambda (m_1) (\log m_1)^{n}
- \sum_{k=1}^{\infty} 2^k \sum_{j=0}^{n} \binom{n}{j} \sum_{m_1 m_2 \leq T/2 \pi} \frac{ (-1)^j (\log m_1)^j a_k(m_2) }{(\log m_1 m_2)^{k-n+j}} \right)
\end{equation*}
up to a power-saving error term, where the coefficients $a_k(m)$ come from $a(m,s)$. 

An application of Perron's formula turns the first sum back into an integral, which we can evaluate the residue of at the pole. Another application of Perron's formula turns the numerator of the inner sum of the second term into an integral that we evaluate the residue of at the pole, and complete the argument with partial summation and summing over $j=,0,\dots,n$ and $k \geq 1$.  

We then outline the proofs of Theorem \ref{thm:FirstMomRho1} in Section \ref{sect:ProofCor1}, of Theorem \ref{thm:ChiRho1} in Section \ref{sect:ProofThmChi1}, and of Theorem \ref{thm:ChiRho} in \ref{sect:ProofThmChi}, leaving the full details of the proofs to the interested reader as they are similar to that of Theorem \ref{thm:FirstMomnthDerivRho1}.

\section{Preliminary Lemmas}\label{sect:PrelimLem}
Define the function $Z_1 (s)$ by
\begin{equation}\label{eq:Z_1}
Z_1(s)=\zeta^{\prime}(s)-\frac{1}{2} \frac{\chi^{\prime}}{\chi}(s) \zeta(s).
\end{equation}
Note that by taking the logarithmic derivative of Hardy's $Z$ function, we see that $Z_1(s)$ is zero exactly when $Z ' (t) = 0$, that is, at an extrema of $|\zeta(s)|$ on the critical line. 

The function $Z_1(s)$ satisfies various properties which were proved by Conrey and Ghosh~\cite{ConGho85} in their lemma. Note that only the properties we use in our proof are listed below. For a proof of this lemma, we refer the reader to their paper.

\begin{lemma}\label{lem:CGLemma}
    We have the following properties for $Z_1(s)$:
    \begin{enumerate}
        \item $|Z_1(1/2 + it)| = |Z'(t)|$.
        \item $Z_1(s)$ satisfies the functional equation
        \[
        Z_1(s) = -\chi(s) Z_1 (1-s)
        \]
        for all $s$, where $\chi (s)$ is the term in the functional equation for the Riemann zeta function, $\zeta (s) = \chi (s) \zeta (1-s)$.
        \item The number of zeros of $Z_1(t)$ up to a height $T$ which lie off the critical line is bounded by $\log T$.
        \item If $\ Z_1(\beta_1 + i \gamma_1) = 0$ then 
        \[
        \left|\beta_1 - \frac{1}{2} \right| \leq \frac{1}{9}
        \]
        for $\gamma_1$ sufficiently large. 
    \end{enumerate}
\end{lemma}

The logarithmic derivative of $Z_1(s)$ is given by
\begin{equation}\label{eq:logderivZ1FE}
    \frac{Z_1'}{Z_1}(s) = \frac{\chi '}{\chi} (s) - \frac{Z_1'}{Z_1}(1-s)
\end{equation}
which follows directly from part (2) in Lemma \ref{lem:CGLemma}.

For $s=\sigma + it$, where $|\sigma| \leq 2$, we have the following estimate for the logarithmic derivative for $\chi (s)$ given by
    \begin{equation}\label{eq:ChiPrimeOverChi}
    \frac{\chi'}{\chi} (s) = - \log \frac{t}{2 \pi} + O \left(\frac{1}{t} \right),
    \end{equation}
(say for $t>1$) which can be found in \cite{ConGho85}, for example. This enables us to rewrite the functional equation above as
\begin{equation}\label{eq:logderivZ1FE2}
    \frac{Z_1'}{Z_1}(s) = - \log \frac{t}{2 \pi} - \frac{Z_1'}{Z_1}(1-s) + O \left(\frac{1}{t} \right).
\end{equation}

\begin{lemma}\label{lem:Dirichletlogderiv}
    For $\Re(s)> 1$ and $t \geq 100$, we have
\[
\frac{Z_1'}{Z_1}(s) = \sum_{m=1}^\infty \frac{a(m,s)}{m^s} + O \left( \frac{1}{t \log t} \right)
\]
where
\begin{equation}\label{eq:a(n,s)}
    a(m,s) = -\Lambda (m) + \sum_{k=1}^{\infty} \frac{1}{f(s)^k} a_k(m)
\end{equation}
with
\begin{equation}\label{eq:f(s)}
f(s) = - \frac{1}{2} \frac{\chi '}{\chi} (s)
\end{equation}
and $a_k(m)$ is of the form
\begin{equation}\label{eq:ak(n)}
a_k(m) = \left((\Lambda \log) \ast \Lambda_{k-1}\right)(m), 
\end{equation}
where we use the notation $\Lambda_k$ to denote 
\[
\underbrace{\Lambda \ast \Lambda \ast \Lambda \ast \dots  \ast \Lambda}_{k-1 \text{ convolutions}}
\]
with the convention that $\Lambda_0(m)$ takes the value $1$ if $m=1$ and $0$ otherwise, and where $\Lambda_1 (m)= \Lambda (m)$ is the usual von Mangoldt function.
\end{lemma}

We give a variation of the proof given in \cite{HugLugPea24}.

\begin{proof}
Recall that we have defined the function $Z_1 (s)$ in \eqref{eq:Z_1} by
\begin{equation*}
Z_1(s)=\zeta^{\prime}(s) + f(s) \zeta(s)
\end{equation*}
where $f(s)$ is given in \eqref{eq:f(s)}. Differentiating this gives
\begin{equation*}
    Z_1'(s) = \zeta ''(s) + f(s) \zeta'(s) + f'(s) \zeta(s).
\end{equation*}
Then 
\begin{equation}\label{eq:logderivZ1}
    \frac{Z_1'}{Z_1}(s) = \frac{\zeta''(s) +  f(s) \zeta'(s)}{\zeta'(s) + f(s) \zeta(s)} + f'(s) \frac{1}{\frac{\zeta'}{\zeta}(s) + f(s)}.
\end{equation}
For $\Re(s)>1$ with $t\geq 1$, the last term can be bounded by $O(1/t \log t)$ by noting that $f'(s) = O(1/t)$ (which follows from \eqref{eq:ChiPrimeOverChi}), $f(s) = O(\log t)$ and $\zeta'(s)/\zeta(s) = O(\log\log t)$ (due to Littlewood \cite{Lit24}, under the Riemann Hypothesis) in the region under consideration. Therefore we can write
\begin{equation*}
    \frac{Z_1'}{Z_1}(s) = \frac{\frac{\zeta''}{\zeta}(s) +  f(s) \frac{\zeta'}{\zeta}(s)}{\frac{\zeta'}{\zeta}(s) + f(s)} + O\left(\frac{1}{t \log t}\right).
\end{equation*}

For $s$ with sufficiently large real part, we can expand the denominator as a geometric series. This gives
\begin{align*}
    \frac{1}{\frac{\zeta'}{\zeta}(s) + f(s) } &= \frac{1}{f(s)}\left(1 + \frac{1}{f(s)}  \frac{\zeta'}{\zeta}(s) \right)^{-1}\\
    &= \frac{1}{f(s)} \sum_{k=0}^\infty \frac{1}{f(s)^k} \left(-\frac{\zeta'}{\zeta}(s)\right)^k.
\end{align*}
This expansion will be valid so long as $\Re(s)$ is large enough so that $\left| \frac{1}{f(s)}  \frac{\zeta'}{\zeta}(s)  \right| \leq 1$. For example, we can take $\Re (s)>1$ and $t \geq 100$.

Then we have
\begin{align}\label{eq:midZ1}
    \frac{Z_1'}{Z_1}(s) &= \frac{1}{f(s)} \left(\frac{\zeta''}{\zeta}(s) +  f(s) \frac{\zeta'}{\zeta}(s) \right) \sum_{k=0}^\infty \frac{1}{f(s)^k} \left(-\frac{\zeta'}{\zeta}(s)\right)^k + O\left(\frac{1}{t \log t}\right) \\
    &= \sum_{n=1}^\infty \frac{a(n,s)}{n^s} + O \left( \frac{1}{t \log t} \right) \notag
\end{align}
where we calculate the coefficients $a(n,s)$ below.

By repeated use of the quotient rule, we may write
\begin{equation*} 
    \frac{d}{ds} \left( \frac{\zeta^{(k)}}{\zeta} (s) \right) = \frac{\zeta^{(k+1)}}{\zeta} (s) - \frac{\zeta'}{\zeta} (s) \frac{\zeta^{(k)}}{\zeta} (s).
\end{equation*}
Rearranging this we have
\begin{equation*} 
    \frac{\zeta^{(k+1)}}{\zeta} (s) = \frac{d}{ds} \left( \frac{\zeta^{(k)}}{\zeta} (s) \right) + \frac{\zeta'}{\zeta} (s) \frac{\zeta^{(k)}}{\zeta} (s).
\end{equation*}
The Dirichlet coefficients of these series then give us the recurrence relation for any non-negative integer $k$,
\begin{equation}\label{eq:RecLam}
\Lambda^{(k+1)} (n) = \Lambda^{(k)} (n) \log (n) + (\Lambda^{(k)} (n) \ast \Lambda) (n),
\end{equation}
for each $n \in \mathbb{N}$, where $\Lambda^{(1)} (n) = \Lambda (n)$ is the usual von Mangoldt function, that is, the coefficient from the Dirichlet series for $-\zeta'(s)/\zeta (s)$. (Note that this notation is not standard - see the remark at the end of this proof.)

Consider now the coefficient from each of the Dirichlet series in \eqref{eq:midZ1}. For the terms on the right-hand side of that equation (ignoring the error term for now), we have
\begin{equation*}
    a(n,s) = \left( \frac{1}{f(s)} \Lambda^{(2)}(n) - \Lambda(n) \right)\sum_{k=0}^\infty \frac{1}{f(s)^k} \Lambda_k(n).
\end{equation*}

Now use the recurrence relation \eqref{eq:RecLam} with $k=1$ to write
\[
\Lambda^{(2)} (n) = ( \Lambda \log + \Lambda_2)(n).
\]
Substitute this into the previous expression to obtain for each $n \in \mathbb{N}$,
\begin{align*}
    a(n,s) &= \left( \frac{1}{f(s)} (\Lambda \log + \Lambda_2)(n) - \Lambda(n) \right)\sum_{k=0}^\infty \frac{1}{f(s)^k} \Lambda_k(n) \\
    &= \sum_{k=0}^{\infty} \frac{1}{f(s)^{k+1}} ((\Lambda \log) \ast \Lambda_{k})(n) + \sum_{k=0}^{\infty} \frac{1}{f(s)^{k+1}} (\Lambda_{2}\ast \Lambda_{k})(n) - \sum_{k=0}^{\infty} \frac{1}{f(s)^{k}} (\Lambda \ast \Lambda_k)(n) \\
    &= \sum_{k=1}^{\infty} \frac{1}{f(s)^{k+1}} ((\Lambda \log) \ast \Lambda_{k})(n) + \sum_{k=1}^{\infty} \frac{1}{f(s)^{k}} \Lambda_{k+1}(n) - \sum_{k=0}^{\infty} \frac{1}{f(s)^{k}} \Lambda_{k+1}(n) \\
    &= \sum_{k=1}^{\infty} \frac{1}{f(s)^{k+1}} ((\Lambda \log) \ast \Lambda_{k})(n) + \sum_{k=0}^{\infty} \frac{1}{f(s)^{k}} \Lambda_{k+1}(n) - \Lambda(n) - \sum_{k=0}^{\infty} \frac{1}{f(s)^{k}} \Lambda_{k+1}(n) \\
    &= \sum_{k=0}^{\infty} \frac{1}{f(s)^{k}} ((\Lambda \log) \ast \Lambda_{k-1})(n) - \Lambda(n), \\
\end{align*}
completing the proof of the Lemma.
\end{proof}

\begin{remark}
    Note that while we can only use the expansion in this Lemma for $t \geq 100$, in the proofs of our results in later sections we will ignore this slight technicality as using the result for $t \geq 1$ in our applications will only introduce an error of $O(1)$, which is well within our other error terms. 
\end{remark}

\begin{remark}
    As our notation $\Lambda_k$ is used for $(k-1)$-convolutions, we use non-standard notation
    \[
    \frac{\zeta^{(k)}}{\zeta}  (s) = \sum_{n=1}^{\infty} \frac{\Lambda^{(k)} (n)}{n^s}.
    \]
    Typically in the literature, $\Lambda_k (n)$ is used here but we have already used it in a different meaning in this section, and want to stay consistent with the notation in \cite{HugLugPea24}. Since this proof is the only place in this section that we use this convention, we don't worry about any confusion or conflicts and remind the reader that we are using $\Lambda_k$ to denote $(k-1)$-convolutions.
\end{remark}

We will be applying the method of stationary phase at various points throughout our argument. The following version of this result can be found in \cite{KaraYil11}, and is based on a similar result due to Gonek \cite{Gonek84} where he took $k$ in the following lemma to be a non-negative integer.
\begin{lemma}\label{lem:StatPhase}
Let $\{b_m\}_{m=1}^{\infty}$ be a sequence of complex numbers such that for any $\varepsilon>0$, $b_m \ll m^\varepsilon$. Let $c>1$ and suppose $|k| = o(\log T)$ as $T \rightarrow \infty$. Then for $T$ sufficiently large,
\[
\frac{1}{2 \pi} \int_{0}^{T} \chi (1-c-it) \left( \log \frac{t}{2 \pi} \right)^k \left( \sum_{m=1}^{\infty} b_m m^{-c -it} \right) dt = \sum_{ m \leq \frac{T}{2 \pi}} b_m (\log m)^k + O\left(T^{c-1/2} (\log T)^k\right).
\]
\end{lemma}

We remind the read that in this paper we will have the convention that the Laurent expansion for $\zeta (s)$ about $s=1$ is given by
\begin{equation*}
\zeta(s) = \frac{1}{s-1} + \gamma_0 - \gamma_1 (s-1) + \frac{1}{2!} \gamma_2 (s-1)^2 + \dots + (-1)^j \frac{1}{j!} \gamma_j (s-1)^j + \dots.
\end{equation*}

The last result we need is from the proof of the main result from Hughes and Pearce-Crump \cite{HugPC22}. This result is unconditional in that paper, but as we are assuming the Riemann Hypothesis here, we will also assume it in the statement of this theorem (this only affects the error term). 

\begin{theorem}\label{thm:HugPC}
Assume the Riemann Hypothesis. For $\rho = \tfrac{1}{2} + i\gamma$ a non-trivial zero of $\zeta (s)$, we have
    \begin{align*}
        &\sum_{0 < \gamma \leq T} \zeta^{(n)} \left( \tfrac{1}{2} + i\gamma \right) \\
        &= (-1)^{n+1}  \sum_{m_1m_2 \leq T/2 \pi} \Lambda (m_1) (\log m_1)^{n} + O\left(T^{1/2 + \varepsilon}\right) \\
        &= (-1)^{n+1} \frac{1}{n+1} \frac{T}{2 \pi} \left( \log  \frac{T}{2 \pi} \right)^{n+1} \\
        &+ (-1)^{n+1} \sum_{k=0}^n \binom{n}{k}  (-1)^ {k}  k!  \left(-1 + \sum_{j=0}^k \frac{1}{j!}\gamma_j \right) \frac{T}{2 \pi} \left( \log  \frac{T}{2 \pi} \right)^{n-k} + n! A_n  \frac{T}{2 \pi}  + O\left(T^{1/2 + \varepsilon}\right)
    \end{align*}
where the $\gamma_j$ are the coefficients in the Laurent expansion of $\zeta (s)$ about $s=1$, given in \eqref{eq:LaurentZeta} and $A_n$ is the $n$th coefficient in the Laurent expansion of $\zeta' (s)/\zeta (s)$ about $s=1$.
\end{theorem}

\section{Proof Theorem \ref{thm:FirstMomnthDerivRho1}}\label{sect:ProofThmDeriv}
\subsection{Initial Manipulations}
We note that the beginning of the proofs of Theorems \ref{thm:FirstMomnthDerivRho1}, \ref{thm:FirstMomRho1}, and \ref{thm:ChiRho1} start in a similar way to that of Conrey and Ghosh \cite{ConGho85}, and of Hughes, Lugmayer and Pearce-Crump \cite{HugLugPea24}. We repeat it here for completeness sake but skip over the details in the later cases.

Under the Riemann Hypothesis, we know that the zeros $\lambda$ of $Z'(t)$ interlace those of $Z(t)$, and so the set of points where $|\zeta (1/2+it)|$ achieves its local maximum are exactly the points where $Z'(t)=0$. 

Let  $\rho_{1} = \beta_{1} + i \gamma_{1}$ be the zeros of $Z_1(s)$. By part (3) of Lemma \ref{lem:CGLemma},  $Z_1(s)$ has $O(\log T)$ zeros off the critical line (that is, $\beta_1 \neq 1/2$). At the zeros off the critical line, we may use part (4) of Lemma \ref{lem:CGLemma} together with the Lindel\"of bounds on $\zeta (s)$ and its derivatives to obtain
\[
\sum_{\substack{0< \gamma_{1} \leq T \\ \beta_1 \neq 1/2}} \zeta^{(n)} (\rho_{1}) \ll T^{1/9+\varepsilon} \log T \ll T^{1/9 + \varepsilon}
\]
for any integer $n \geq 0$. 

Therefore, we have
\begin{equation*}
    \sum_{0< \lambda \leq T} \zeta^{(n)} \left( \frac{1}{2} + i\lambda \right) =  \sum_{\substack{0 < \gamma_1 \leq T \\ \beta_1 = 1/2}} \zeta^{(n)} \left( \frac{1}{2} + i\gamma_1 \right).
\end{equation*}

We can write this as an integral using Cauchy's theorem, 
\begin{align}\label{eq:Cauchy}
\frac{1}{2\pi i} \int_{\mathcal{C}} \frac{Z_1'}{Z_1}(s) \zeta^{(n)} (s) \ ds &= \sum_{0 < \gamma_1 \leq T} \zeta^{(n)} (\rho_1) \notag  \\
&= \sum_{\substack{0 < \gamma_1 \leq T \\ \beta_1 = 1/2}} \zeta^{(n)} \left( \frac{1}{2} + i\gamma_1 \right) + O\left(T^{1/9 + \varepsilon} \right)
\end{align}
where $\mathcal{C}$ is a positively oriented contour with vertices $c + i$, $c + iT$, $1-c+iT$, and $1-c +i$, where $c=1+1/\log T$. We may assume, without loss of generality, that the distance from the contour to any zero $\rho_1$ of $Z_1(s)$ is uniformly $\gg 1/\log T$.

We split the integral as 
\begin{align*}
&\frac{1}{2 \pi i} \left( \int_{c+i}^{c+iT} + \int_{c+iT}^{1-c+iT} + \int_{1-c+iT}^{1-c+i} + \int_{1-c+i}^{c+i} \right) \frac{Z_1 '}{Z_1}(s) \zeta ^{(n)} (s) \ ds \\
&= S^R + S^T + S^L + S^B,
\end{align*}
say. We will first bound $S^B, S^T, S^R$ trivially within error term, and show that the leading contributions to the asymptotic will then come from $S^L$. Consider each of these integrals in turn. 

First note that the integral $S^B$, which is the integral along the bottom of the contour, is $O(1)$. 

For the integral $S^T$, which is the integral along the top of the contour, we may use the bound proved in \cite{HugLugPea24} that states if $|T-\rho_1| \gg 1/\log T$ for all zeros $\rho_1$ of $Z_1$ (which we are in this proof), we have
\[
\frac{Z_1'}{Z_1}(\sigma + i T) \ll (\log T)^2
\]
uniformly for $-1< \sigma \leq 2$. Combining the Lindel\"of bounds on $\zeta (s)$ and its derivatives with this shows that the integral along the top of the contour is $O\left(T^{1/2 + \varepsilon} \right)$. 

Next note that for the integral $S^R$, which is the integral along the right-hand vertical side of the contour, we have $c>1$ which is past the abscissa of convergence for $\zeta^{(n)} (s)$ and $Z_1(s)$ as given in \eqref{eq:Z_1}. Then 
\[
S^R = \frac{1}{2 \pi} \int_{0}^{T} \frac{Z_1'}{Z_1}(c+it) \zeta^{(n)}(c+it) \ dt \ll T^\varepsilon.
\]

Note then that $S^B, S^T, S^R$ are all within an error term of $O\left(T^{1/2+\varepsilon}\right)$. All that remains is to evaluate $S^L$. Observe that
\[
S^L = \frac{1}{2 \pi i} \int_{1-c+iT}^{1-c+i} \frac{Z_1 '}{Z_1}(s) \zeta ^{(n)} (s) \ ds = -\frac{1}{2 \pi i} \int_{c-iT}^{c-i} \frac{Z_1 '}{Z_1}(1-s) \zeta ^{(n)} (1-s) \ ds = - \overline{I},
\]
where
\begin{equation}
I = \frac{1}{2 \pi i} \int_{c+i}^{c+iT} \frac{Z_1 '}{Z_1}(1-s) \zeta ^{(n)} (1-s) \ ds. \label{eq:II}
\end{equation}

Overall, we have 
\begin{equation*}
    \sum_{0< \lambda \leq T} \zeta^{(n)} \left( \frac{1}{2} + i\lambda \right) =  - \overline{I} + O\left(T^{1/2+\varepsilon}\right).
\end{equation*}

\subsection{Deriving the main terms}
To begin manipulating $I$ into a form that we can evaluate, we use the logarithmic derivative of the functional equation for $Z_1 (s)$ given in \eqref{eq:logderivZ1FE}. We will also use for functional equation for the derivatives of $\zeta (1-s)$, which can be derived easily from the functional equation for $\zeta (s)$, given in \ref{eq:FE}. Substituting these two expressions into \eqref{eq:II} gives
\begin{align*}
    I &= \frac{1}{2 \pi i} \int_{c+i}^{c+iT} \left( - \log \frac{t}{2\pi} \right) (-1)^n \chi (1-s) \sum_{j=0}^n \binom{n}{j} \left(  \log \frac{t}{2 \pi} \right)^{n-j} \zeta^{(j)} (s) \ ds \\
    &\quad + \frac{1}{2 \pi i} \int_{c+i}^{c+iT} \left( - \frac{Z_1 '}{Z_1} (s) \right) (-1)^n \chi (1-s) \sum_{j=0}^n \binom{n}{j} \left( \log \frac{t}{2 \pi} \right)^{n-j} \zeta^{(j)} (s) \ ds + O \left( T^{1/2 + \varepsilon } \right).
\end{align*}
Simplifying slightly gives
\begin{align*}
    I &= \frac{(-1)^{n+1}}{2 \pi i} \int_{c+i}^{c+iT} \chi (1-s) \sum_{j=0}^n \binom{n}{j} \left( \log \frac{t}{2 \pi} \right)^{n-j+1} \zeta^{(j)} (s) \ ds \\
    &\quad +\frac{(-1)^{n+1}}{2 \pi i} \int_{c+i}^{c+iT}  \chi (1-s) \sum_{j=0}^n \binom{n}{j} \left( \log \frac{t}{2 \pi} \right)^{n-j} \zeta^{(j)} (s) \frac{Z_1 '}{Z_1} (s) \ ds + O \left( T^{1/2 + \varepsilon } \right) \\
    &= I_1 + I_2 + O \left( T^{1/2 + \varepsilon } \right),
\end{align*}
say. 

\subsection{The integral $I_1$}
First consider $I_1$. We want to apply Lemma \ref{lem:StatPhase} here to extract terms for the asymptotic. To do this we begin by switching the order of integration and summation and substitute the Dirichlet series for $\zeta^{(j)} (s) $ to obtain
\begin{equation*}
    I_1 = (-1)^{n+1} \sum_{j=0}^n \binom{n}{j} \left[ \frac{1}{2 \pi i} \int_{c+i}^{c+iT} \chi (1-s) \left( \log \frac{t}{2 \pi} \right)^{n-j+1} \sum_{m=1}^{\infty} \frac{(-1)^j (\log m)^j}{m^s} \ ds \right]
\end{equation*}
Now apply Lemma \ref{lem:StatPhase} to the integral, obtaining
\begin{equation*}
    I_1 = (-1)^{n+1} \sum_{j=0}^n \binom{n}{j} \left[ \sum_{m \leq T/ 2\pi} (\log m)^{n-j+1} (-1)^j \log(m)^j \right] + O \left( T^{1/2 + \varepsilon } \right).
\end{equation*}
Simplifying this gives
\begin{equation*}
    I_1 = (-1)^{n+1} \sum_{j=0}^n \binom{n}{j} (-1)^j \left[ \sum_{m \leq T/ 2\pi} (\log m)^{n+1} \right] + O \left( T^{1/2 + \varepsilon } \right).
\end{equation*}
Now by the Binomial Theorem, this summation over $j$ equals zero for $n>0$ and equals $1$ if $n=0$. Hence in the case $n=0$, we have a contribution of 
\begin{equation*}
    I_1 = - \sum_{m \leq T/ 2\pi} \log m+ O \left( T^{1/2 + \varepsilon } \right).
\end{equation*}

By standard methods, we have 
\[
I_1 = -\frac{T}{2\pi} \log \frac{T}{2\pi} + \frac{T}{2\pi} + O \left( T^{1/2 + \varepsilon}\right).
\]
and so we have proved the following result.

\begin{lemma}\label{lem:I_1}
    For $n$ a non-negative integer, the integral $I_1$ satisfies
    \begin{equation*}
    I_1 =
    \begin{dcases*}
        -\frac{T}{2\pi} \log \frac{T}{2\pi} + \frac{T}{2\pi} + O \left( T^{1/2 + \varepsilon}\right)
   & \text{ if } n=0  \\
        O \left( T^{1/2 + \varepsilon} \right) 
   & \text{ if } n > 0.
    \end{dcases*}
    \end{equation*}
\end{lemma}

\subsection{The integral $I_2$}
Similarly to when we evaluated the integral $I_1$, we want to apply stationary phase calculations to extract terms for the asymptotic from $I_2$. To do this we begin by switching the order of integration and summation and substitute the Dirichlet series for $\zeta^{(j)} (s) $ and the series for the logarithmic derivative of $Z_1 (s)$ to obtain
\begin{align*}
    I_2 =& (-1)^{n+1} \sum_{j=0}^n \binom{n}{j} \\ 
    &\quad \times \left[ \frac{1}{2 \pi i} \int_{c+i}^{c+iT}  \chi (1-s)  \left( \log \frac{t}{2 \pi} \right)^{n-j} \sum_{m_1=1}^{\infty} \frac{(-1)^j (\log m_1)^j}{m_1^s} \sum_{m_2=1}^{\infty} \frac{(-\Lambda (m_2))}{m_2^s} \ ds \right] \\
    &+ \sum_{k=1}^{\infty} (-1)^{n+1} \sum_{j=0}^n \binom{n}{j} \\ 
    &\quad \times \left[ \frac{1}{2 \pi i} \int_{c+i}^{c+iT}  \chi (1-s) \frac{1}{f(s)^k}  \left( \log \frac{t}{2 \pi} \right)^{n-j} \sum_{m_1=1}^{\infty} \frac{(-1)^j (\log m_1)^j}{m_1^s}  \sum_{m_2=1}^{\infty}\frac{a_k(m_2))}{m_2^s} \ ds \right] \\
    &= I_{2,1} + I_{2,2},
\end{align*}
say, where we split the initial term of $Z_1'(s)/Z_1 (s)$ from the rest as $I_{2,1}$ follows directly from stationary phase methods, while $I_{2,2}$ has a more involved argument.

For $I_{2,1}$ we may use Lemma \ref{lem:StatPhase} to write
\[
I_{2,1} = (-1)^{n+1} \sum_{j=0}^n \binom{n}{j} \left[ \sum_{m_1m_2 \leq T/2 \pi} (-1)^j (\log m_1)^j (- \Lambda (m_2)) (\log m_1 m_2)^{n-j}  \right] + O \left( T^{1/2 + \varepsilon}\right).
\]
Switching the order of summation, expanding the logarithms, and performing some basic algebraic manipulations allows us to rewrite this as
\[
I_{2,1} = (-1)^{n}  \sum_{m_1m_2 \leq T/2 \pi} \Lambda (m_1) (\log m_1)^{n} + O \left( T^{1/2 + \varepsilon}\right).
\]

We note that this calculation was preformed in Hughes and Pearce-Crump \cite{HugPC22}. It is, in essence, the negative of the result given in Theorem \ref{thm:HugPC}. That is, 
\begin{multline}\label{eq:I21}
    I_{2,1} = (-1)^{n} \frac{1}{n+1} \frac{T}{2 \pi} \left( \log  \frac{T}{2 \pi} \right)^{n+1}  \\
    + (-1)^{n} \sum_{k=0}^n \binom{n}{k}  (-1)^ {k}  k!  \left(-1 + \sum_{j=0}^k \frac{1}{j!}\gamma_j \right) \frac{T}{2 \pi} \left( \log  \frac{T}{2 \pi} \right)^{n-k}  \\
    - n! A_n  \frac{T}{2 \pi}  + O\left(T^{1/2 + \varepsilon}\right) 
\end{multline}
where we remind the reader that the $\gamma_j$ are the coefficients from the Laurent expansion of $\zeta (s)$ about $s=1$ and the $A_n$ are the coefficients from the Laurent expansion of $\zeta'(s)/\zeta (s)$ about $s=1$.

Finally we need to evaluate $I_{2,2}$. By \eqref{eq:ChiPrimeOverChi} and \eqref{eq:f(s)}, we can write
\begin{equation*}
f(s) = -\frac{1}{2} \frac{\chi'}{\chi} (s) = \frac{1}{2} \log \frac{t}{2 \pi} + O\left( \frac{1}{t} \right)
\end{equation*}
and so
\begin{equation*}
    \frac{1}{f(s)^k} = \frac{2^k}{(\log t/2\pi)^k} \left( 1 + O\left( \frac{1}{t \log t} \right) \right)^{-k} = \frac{2^k}{(\log t/2\pi)^k} \left( 1 + O\left( \frac{k}{t \log t} \right) \right).
\end{equation*}
Then 
\begin{align*}
I_{2,2} &= \sum_{k=1}^{\infty} (-1)^{n+1} \sum_{j=0}^n \binom{n}{j} \\
&\hphantom{<lotsoftext>} \times \frac{1}{2 \pi i} \int_{c+i}^{c+iT}  \chi (1-s) \frac{2^k}{(\log t/2\pi)^k} \left( 1 + O\left( \frac{k}{t \log t} \right) \right) \\
&\hphantom{<lotsoftextlotsoftext>} \times \left( \log \frac{t}{2 \pi} \right)^{n-j} \sum_{m_1=1}^{\infty} \frac{(-1)^j (\log m_1)^j}{m_1^s}  \sum_{m_2=1}^{\infty}\frac{a_k(m_2))}{m_2^s} \ ds \\
&= \sum_{k=1}^{\infty} (-1)^{n+1} \sum_{j=0}^n \binom{n}{j} \\
&\hphantom{<lotsoftext>} \times \frac{1}{2 \pi i} \int_{c+i}^{c+iT}  \chi (1-s) \frac{2^k}{(\log t/2\pi)^k}  \left( \log \frac{t}{2 \pi} \right)^{n-j} \\
&\hphantom{<lotsoftextlotsoftext>} \times\sum_{m_1=1}^{\infty} \frac{(-1)^j (\log m_1)^j}{m_1^s}  \sum_{m_2=1}^{\infty}\frac{a_k(m_2)}{m_2^s} \ ds + O\left( T^{1/2 + \varepsilon} \right).
\end{align*}

After these initial manipulations performed above, we may apply Lemma \ref{lem:StatPhase}. This gives
\begin{align}
    I_{2,2} &= (-1)^{n+1} \sum_{k=1}^{\infty} 2^k \sum_{j=0}^{n} \binom{n}{j} \left[ \sum_{m_1 m_2 \leq T/2 \pi} \frac{ (-1)^j (\log m_1)^j a_k(m_2) (\log m_1 m_2)^{n-j}}{(\log m_1 m_2)^k} \right] \notag \\
    &\qquad + O\left(T^{1/2 + \varepsilon}\right) \label{eq:I22}.
\end{align}

\begin{remark}
Now we note that we could repeat our trick that we used when switching the order of summation and expanding the logarithm when we evaluated $I_{2,1}$. However it is then unclear what the correct Dirichlet series is to use in the upcoming Perron argument. Instead we take the approach of cancelling the appropriate powers of the logarithm, and work with $I_{2,2}$ in this form. 
\end{remark}

Switching the order of summation in $I_{2,2}$ gives 
\begin{equation}\label{eq:I2,2}
    I_{2,2} = (-1)^{n+1} \sum_{k=1}^{\infty} 2^k \sum_{j=0}^{n} \binom{n}{j} \left[ \sum_{m_1 m_2 \leq T/2 \pi} \frac{ (-1)^j (\log m_1)^j a_k(m_2)}{(\log m_1 m_2)^{k-n+j}} \right] + O\left(T^{1/2 + \varepsilon}\right)
\end{equation}

To evaluate $I_{2,2}$, we begin with evaluating the numerator of the inner sum. After calculating the sum of the numerator of the inner sum, we can preform partial summation to reinsert the logarithm in the denominator of the inner sum, and then sum over $j$ and $k$. 

\begin{lemma}\label{lem:Ak}
    Let 
    \[
    A_{k,j}(x) = \sum_{m \leq x} \alpha_{k,j}(m),
    \]
    where 
    \[
    \alpha_{k,j}(m) = \sum_{m_1 m_2 = m} (-1)^j (\log m_1)^j a_k(m_2)
    \]
    and where $a_k(m_2)$ is given in \eqref{eq:ak(n)}.
    This sum can also be written as
    \[
    A_{k,j}(x) = \sum_{1 \leq m n_1 n_2 ... n_k  \leq x} (-1)^j(\log m)^j \log (n_1) \Lambda (n_1) \Lambda (n_2) \dots \Lambda (n_k) .
    \]
    Then for large $x$,
    \[
    A_{k,j}(x) = x \sum_{\ell=0}^{k+j+1} \frac{c_{\ell}^{k,j}}{(k+j+1-\ell)!} (\log x)^{k+j+1-\ell} + O\left(x^{1/2+\varepsilon}\right)
    \]
    where the $c_{\ell}^{k,j}$ are the Laurent series coefficients around $s=1$ of
    \[
    \left(\frac{\zeta '}{\zeta} (s) \right)' \left(-\frac{\zeta '}{\zeta} (s) \right)^{k-1} \zeta^{(j)} (s) \frac{1}{s} = \sum_{\ell=0}^\infty c_{\ell}^{k,j} (s-1)^{-k-j-2+\ell} .
    \]
\end{lemma}
    
\begin{remark}
    A long but straightforward calculation tells us what the coefficients $c_{\ell}^{k,j}$ equal. The leading coefficient is when $\ell = 0$ and is given by
    \[
    c_{0}^{k,j} = (-1)^j j!
    \]
    for all $j \geq 0$. When $\ell = 1$, the subleading coefficients are given by
    \begin{align*}
        c_{1}^{k,j} = 
        \begin{cases}
        -1 + \gamma_0 + (1-k) \gamma_0 &\text{ if $j=0$} \\
        (-1)^j j! (-1 + (1-k) \gamma_0) &\text{ for all $j \geq 1$.} 
        \end{cases}
    \end{align*}
\end{remark}

\begin{remark}
    The superscripts on the coefficients $c_{\ell}^{k,j}$ is to indicate that the terms depend on both $k$ and $j$. 
\end{remark}

\begin{proof}
    We begin by applying Perron's formula to $A_{k,j}(x)$, that is, we have 
\begin{equation*}
    A_{k,j}(x) = \frac{1}{2 \pi i} \int_{3/2-i R}^{3/2+i R} \left(\frac{\zeta '}{\zeta} (s) \right)' \left(-\frac{\zeta '}{\zeta} (s) \right)^{k-1} \zeta^{(j)} (s) \frac{x^s}{s} \ ds + O\left(\frac{x^{3/2+\varepsilon}}{R}\right)
\end{equation*}

Note that the integrand has a pole of order $k+j+2$ at $s=1$. We shift past the pole at $s=1$ to the line $\Re(s)=1/2+\varepsilon$, and evaluate the residue at $s=1$. Therefore, up to a power-saving error, $A_{k,j}(x)$ will equal the residue of the integrand around that point.  That is, 
\[
A_{k,j}(x) = x \sum_{\ell=0}^{k+j+1} b_{\ell}^{k,j} (\log x)^{k+j+1-\ell} + O\left(R^{\varepsilon} x^{1/2+\varepsilon}\right) + O\left(\frac{x^{3/2+\varepsilon}}{R}\right)
\]
for some constants $b_{\ell}^{k,j}$. We have started the sum at the largest power of $\log x$ as this simplifies future calculations. Choosing $R=x$ optimises the error terms, giving an error of $O(x^{1/2+\varepsilon})$.

To evaluate the residue, we use the various Laurent expansions about $s=1$ in the various terms in the integrand. The expansions we need are
\begin{align*}
    &\left( \frac{\zeta'}{\zeta}(s)\right)' &=\quad& \frac{1}{(s-1)^2}+\left(-2 \gamma_1-\gamma_0^2\right)+ \left(6 \gamma_0  \gamma_1+3 \gamma_2+2 \gamma_0^3\right) (s-1) + \dots \\
    & \left(- \frac{\zeta'}{\zeta}(s)\right)^{k-1} &=\quad& \frac{1}{(s-1)^{k-1}} +\frac{(1-k) \gamma_0}{(s-1)^{k-2}} + \frac{\frac12k(k-1)\gamma_0^2 + 2(k-1)\gamma_1}{(s-1)^{k-3}} + \dots\\
    &\quad  \zeta^{(j)} (s) &=\quad& \frac{(-1)^j j!}{(s-1)^{j+1}} + (-1)^j \gamma_j + (-1)^{j+1} \gamma_{j+1} (s-1) + \dots\\
    &\qquad    \frac{1}{s}&=\quad& 1 - (s-1) + (s-1)^2 + ...
\end{align*}
where the Stieltjes constants are given in \eqref{eq:LaurentZeta}.

Combining these Laurent expansions allows us to write 
    \[
    \left(\frac{\zeta '}{\zeta} (s) \right)' \left(-\frac{\zeta '}{\zeta} (s) \right)^{k-1} \zeta^{(j)} (s) \frac{1}{s} = \sum_{\ell=0}^\infty c_{\ell}^{k,j} (s-1)^{-k-j-2+\ell}
    \]
where we have explicitly given the first few values of $c_{\ell}^{k,j}$ in the remark following the statement of the lemma. 

The residue is the coefficient of $1/(s-1)$ which will come from combining these terms with the $x^s$ in the integrand (which is the only term that contains an $x$). Its expansion about $s=1$ is
\[
x^s = x \left( 1 + (s-1) \log x + \frac{(s-1)^2}{2!} (\log x)^2 + ... + \frac{(s-1)^k}{k!} (\log x)^k +...\right) 
\]
so the coefficient of $(\log x)^{k+j+1-\ell}$ in the $(s-1)^{-1}$ term equals
\[
b_{\ell}^{k,j} = \frac{c_{\ell}^{k,j}}{(k+j+1-\ell)!}
\]
and such a combination is possible for $\ell$ between $0$ and $k+j+1$.
\end{proof}

To complete the proof of Theorem \ref{thm:FirstMomnthDerivRho1}, we want to use the asymptotic from Lemma \ref{lem:Ak} in \eqref{eq:I2,2} and use partial summation to reinsert the logarithm in the denominator of the inner sum in \eqref{eq:I22}.

This means that we need to calculate 
\[
I_{2,2} = (-1)^{n+1} \sum_{k=1}^{\infty} 2^k \sum_{j=0}^{n} \binom{n}{j} \left[ A_k \left( \frac{T}{2 \pi} \right) f\left( \frac{T}{2 \pi} \right) - \int_{2}^{T/2\pi} A_k(x) f'(x) \ dx \right] + O\left(T^{1/2 + \varepsilon}\right)
\]
with $f(x) = 1/(\log x)^{k-n+j}$, so that $f'(x) = -(k-n+j)/(x (\log x)^{k-n+j+1})$, and with $A_{k,j} (x)$ as in Lemma \ref{lem:Ak}. 

Note that while we want to sum the inner sum in $I_{2,2}$ over all positive integers $m_1,m_2$ such that $1 \leq m_1m_2 \leq T/2\pi$, this introduces a problem with the logarithms appearing in the denominator. Instead we will sum from $2$ to avoid this problem, at the cost of a $O(1)$ error which we ignore as it is smaller than our largest error. 

Then by partial summation,
\begin{align}
    I_{2,2} &= (-1)^{n+1} \sum_{k=1}^{\infty} 2^k \sum_{j=0}^{n} \binom{n}{j} \notag \\ 
    &\hphantom{spacespace} \left[ \left( \frac{T}{2 \pi} \right) \sum_{\ell=0}^{k+j+1} \frac{c_{\ell}^{k,j}}{(k+j+1-\ell)!} L^{n+1-\ell} \right. \notag \\
    &\hphantom{spacespacespacespace} + \left. (k-n+j) \sum_{\ell=0}^{k+j+1} \frac{c_{\ell}^{k,j}}{(k+j+1-\ell)!} \int_{2}^{T/2\pi}  (\log x)^{n-\ell} \ dx \right]\notag \\
    &\hphantom{spacespacespacespacespacespacespacespace}+ O\left(T^{1/2 + \varepsilon}\right) \label{eq:I221}
\end{align}
where we write $L=\log T/2\pi$.

\begin{lemma}\label{lem:J1J2}
    Writing $L = \log T/2\pi$, for any integer $K \geq 1$, we have
    \begin{align*}
    I_{2,2} &= (-1)^{n+1} \frac{T}{2\pi} \sum_{k=1}^{\infty} 2^k \sum_{j=0}^{n} \binom{n}{j} \left( L^{n+1} \frac{c_{0}^{k,j}}{(k+j+1)!} \right. \\
    &\hphantom{spacespa}+\sum_{\ell = 0}^{n} L^{n-\ell} \left[   \frac{c_{\ell+1}^{k,j}}{(k+j-\ell)!} + (k-n+j) \frac{c_{\ell}^{k,j}}{(k+j+1-\ell)!}  \right.\\
    &\hphantom{spacespacespa} \left. \left. + (k-n+j) \sum_{m=0}^{\ell-1} (-1)^{\ell -m} (n-m) \dots (n-\ell+1) \frac{c_{m}^{k,j}}{(k+j+1-m)!}  \right]  \right) \\
    &+(-1)^{n+1} \frac{T}{2\pi} \sum_{m=1}^{K} \frac{1}{L^{m}} \sum_{k=1}^{\infty} 2^k \sum_{j=0}^{n} \binom{n}{j} \beta_{m}^{k,j} + O \left( \frac{T}{L^{K+1}} \right),
    \end{align*}
    where for $1 \leq m \leq k+j-n$, we have
    \[
    \beta_{m}^{k,j} = \frac{c_{m+n+1}^{k,j}}{(k+j-n-m)!} + \frac{(m-1)!}{k-n+j-1} \sum_{\ell=0}^{m-1}  \binom{k-n+j}{\ell} c_{\ell+n+1}^{k,j}
    \] 
    and where for $m \leq k+j-n+1$, we have
    \[
    \beta_{m}^{k,j} = \frac{(m-1)!}{k-n+j-1} \sum_{\ell=0}^{k+j-n}  \binom{k-n+j}{\ell} c_{\ell+n+1}^{k,j},
    \]
    where the constants $c_{\ell}^{k,j}$ are given in Lemma \ref{lem:Ak}.
\end{lemma}

\begin{proof}
    We begin by noting from \eqref{eq:I221} that we can obtain the leading order, that is, the $L^{n+1}$, from the first sum within the square brackets by setting $\ell = 0$. There are then positive and negative powers of $L^{n}$ which come from both sums in the square brackets. 

    We begin by splitting the sums into those that contribute to non-negative powers of $L$, and those that contribute to negative powers of $L$. This is because we can evaluate the positive powers exactly, while the negative powers are slightly more complicated. The positive powers occur for $\ell \leq n+1$ in the first sum, and $\ell \leq n$ in the second sum from the square brackets. This splitting gives
    \begin{align*}
        I_{2,2} &= (-1)^{n+1} \sum_{k=1}^{\infty} 2^k \sum_{j=0}^{n} \binom{n}{j} \notag \\ 
        &\hphantom{spacespace} \left[ \left( \frac{T}{2\pi} \sum_{\ell = 0}^{n+1} \frac{c_{\ell}^{k,j}}{(k+j+1-\ell)!} L^{n+1-\ell} \right. \right. \notag \\
        &\hphantom{spacespacespacespace} \left. + (k-n+j) \sum_{\ell=0}^{n} \frac{c_{\ell}^{k,j}}{(k+j+1-\ell)!}  \int_{2}^{T/2\pi} (\log x)^{n-\ell} \ dx \right) \notag \\
        &\hphantom{spacespace} \left. + \left( \frac{T}{2\pi} \sum_{\ell = n+2}^{k+j+1} \frac{c_{\ell}^{k,j}}{(k+j+1-\ell)!} L^{n+1-\ell} \right. \right. \notag \\
        &\hphantom{spacespacespacespace} \left. \left. + (k-n+j) \sum_{\ell=n+1}^{k+j+1} \frac{c_{\ell}^{k,j}}{(k+j+1-\ell)!}  \int_{2}^{T/2\pi} (\log x)^{n-\ell} \ dx \right) \notag \right] \\
        &\hphantom{spacespacespacespacespacespacespacespace}+ O\left(T^{1/2 + \varepsilon}\right)
    \end{align*}
     
For ease of notation, we write the sum $I_{2,2}$ as 
\begin{equation*}
    I_{2,2} = (-1)^{n+1} \sum_{k=1}^{\infty} 2^k \sum_{j=0}^{n} \binom{n}{j} \left[ J_1 + J_2 \right] + O\left(T^{1/2 + \varepsilon}\right),
\end{equation*}
say, where both $J_1$ and $J_2$ depend on $k$ and $j$, and the $J_1$ sum is for the non-negative powers of the logarithm, and the $J_2$ sum is for the negative powers of the logarithm.

We now consider the two sums $J_1$ and $J_2$ in turn. 

Consider $J_1$ first. Recall that for $m \geq 0$, we have
\[
\int_{2}^{T/2\pi} (\log x)^m \ dx = \frac{T}{2 \pi} \sum_{r=0}^{m} (-1)^r \frac{m!}{(m-r)!} L^{m-r} + O(1).
\]
Then using this (and ignoring the error term as it is consumed by larger error terms that we have already seen in this proof), we have
\begin{align*}
    J_1 & =\frac{T}{2\pi}\left[  \sum_{\ell = 0}^{n+1} \frac{c_{\ell}^{k,j}}{(k+j+1-\ell)!} L^{n+1-\ell} \right. \\
    &\hphantom{spacespace} \left. + (k-n+j) \sum_{\ell=0}^{n} \frac{c_{\ell}^{k,j}}{(k+j+1-\ell)!} \left(  \sum_{r=0}^{n-\ell} (-1)^r \frac{(n-\ell)!}{(n-\ell-r)!} L^{n-\ell-r} \right) \right].
\end{align*}
We first note that we can extract the leading order behaviour from the first sum, and reindex the remaining terms in the first sum to obtain
\begin{align*}
    J_1 & =\frac{T}{2\pi} L^{n+1} \frac{c_{0}^{k,j}}{(k+j+1)!} \\
    &\hphantom{spacespa}+ \frac{T}{2\pi}\left[  \sum_{\ell = 0}^{n} \frac{c_{\ell+1}^{k,j}}{(k+j-\ell)!} L^{n-\ell} \right.\\
    &\hphantom{spacespacespa} \left. + (k-n+j) \sum_{\ell=0}^{n} \frac{c_{\ell}^{k,j}}{(k+j+1-\ell)!} \left(  \sum_{r=0}^{n-\ell} (-1)^r \frac{(n-\ell)!}{(n-\ell-r)!} L^{n-\ell-r} \right) \right].
\end{align*}
We now reorder the second summation to write
\begin{align*}
    J_1 & =\frac{T}{2\pi} L^{n+1} \frac{c_{0}^{k,j}}{(k+j+1)!} \\
    &\hphantom{spacespa}+ \frac{T}{2\pi} \sum_{\ell = 0}^{n} \left[   \frac{c_{\ell+1}^{k,j}}{(k+j-\ell)!} + (k-n+j) \frac{c_{\ell}^{k,j}}{(k+j+1-\ell)!}  \right.\\
    &\hphantom{spacespacespa} \left. + (k-n+j) \sum_{m=0}^{\ell-1} (-1)^{\ell -m} (n-m) \dots (n-\ell+1) \frac{c_{m}^{k,j}}{(k+j+1-m)!}  \right] L^{n-\ell}
\end{align*}
where if $\ell=0$ the last sum is empty.

We now calculate the contribution from $J_2$. By a simple relabelling, we have
\begin{align*}
    J_2 &= \frac{T}{2\pi} \sum_{\ell = 1}^{k+j+1} \frac{c_{\ell+n+1}^{k,j}}{(k+j-n-\ell)!} L^{-\ell} \\
    &\hphantom{spacespacespacespace} + (k-n+j) \sum_{\ell=1}^{k+j+1-n} \frac{c_{\ell+n}^{k,j}}{(k+j+1-n-\ell)!}  \int_{2}^{T/2\pi} (\log x)^{-\ell} \ dx,
\end{align*}
so both of our sums in $J_2$ start from $\ell=1$. 

Next we note that in the remaining integrals, all of the powers of $\log x$ are negative, so evaluate to an infinite chain of descending powers of $L$. For example, for $1 \leq m \leq M$, we have
\begin{align*}
    \int_{2}^{T/2\pi} \frac{1}{(\log x)^m} \ dx & = \frac{T}{2\pi}\frac{1}{L^m} + m \int_{2}^{T/2\pi} \frac{1}{(\log x)^{m+1}} \ dt + O(1)\\
    &=  \frac{T}{2\pi} \sum_{r=m}^{M} \frac{(r-1)!}{(m-1)!}\frac{1}{L^{r}} 
   + O\left( \frac{T}{L^{M+1}} \right).
\end{align*}

We write $\beta_{m}^{k,j}$ for the coefficient of $L^{-m}$ and split into two cases.

If $m \leq k+j-n$, then there will be a contribution from the first sum in $J_2$ when $m=\ell$ and from all integrals in the second sum for $\ell = 1,\dots, m$. We have
\[
\beta_{m}^{k,j} = \frac{c_{m+n+1}^{k,j}}{(k+j-n-m)!} + (k-n+j) \sum_{\ell=1}^{m} \frac{c_{\ell+n}^{k,j}}{(k+j+1-n-\ell)!} \frac{(m-1)!}{(\ell -1)!}.
\]
Using the fact
\[
\frac{(k-n+j)}{(k+j+1-n-\ell)!(\ell -1)!} = \binom{k-n+j}{\ell-1} \frac{1}{k-n+j-1},
\]
this simplifies to 
\[
\beta_{m}^{k,j} = \frac{c_{m+n+1}^{k,j}}{(k+j-n-m)!} + \frac{(m-1)!}{k-n+j-1} \sum_{\ell=1}^{m} \binom{k-n+j}{\ell-1} c_{\ell+n}^{k,j} .
\]

If $m \geq k+j-n+1$, then there will only be a contribution coming from the second sum of $J_2$, but every term in that sum contributes. We have
\[
\beta_{m}^{k,j} = \frac{(m-1)!}{k-n+j-1} \sum_{\ell=1}^{k+j+1-n}  \binom{k-n+j}{\ell-1} c_{\ell+n}^{k,j}.
\]
After a trivial relabelling for both versions of $\beta_{m}^{k,j}$, we have the statement of the lemma.
\end{proof}

We can see that we have
\begin{align*}
    I_{2,2} = \mathfrak{a}_{n+1} \frac{T}{2\pi}L^{n+1} + \frac{T}{2\pi} \sum_{\ell=0}^{n} \mathfrak{a}_{n-\ell} L^{n-\ell} + \frac{T}{2\pi} \sum_{m=1}^{K} \frac{\mathfrak{b}_m}{L^{m}} + O\left(\frac{T}{L^{K+1}} \right),
\end{align*}
where the leading order coefficient is given by
    \[
    \mathfrak{a}_{n+1} = (-1)^{n+1} \sum_{k=1}^{\infty} 2^k \sum_{j=0}^{n} \binom{n}{j} \frac{c_{0}^{k,j}}{(k+j+1)!},
    \]
the subleading, non-negative logarithm power coefficients are given by
    \begin{align*}
        \mathfrak{a}_{n-\ell} = (-1)^{n+1}& \sum_{k=1}^{\infty} 2^k \sum_{j=0}^{n} \binom{n}{j} \left[ \frac{c_{\ell+1}^{k,j}}{(k+j-\ell)!} + (k-n+j) \frac{c_{\ell}^{k,j}}{(k+j+1-\ell)!} \right. \\
        &+ \left. (k-n+j) \sum_{m=0}^{\ell-1} (-1)^{\ell -m} (n-m) \dots (n-\ell+1) \frac{c_{m}^{k,j}}{(k+j+1-m)!}  \right],
    \end{align*}
and the subleading, negative logarithm power coefficients are given by
\begin{align*}
    \mathfrak{b}_{m} = (-1)^{n+1}& \sum_{k=1}^{\infty} 2^k \sum_{j=0}^{n} \binom{n}{j} \beta_{m}^{k,j},
    \end{align*}
    where for $1 \leq m \leq k+j-n$, we have
    \[
    \beta_{m}^{k,j} = \frac{c_{m+n+1}^{k,j}}{(k+j-n-m)!} + \frac{(m-1)!}{k-n+j-1} \sum_{\ell=0}^{m-1}  \binom{k-n+j}{\ell} c_{\ell+n+1}^{k,j}
    \] 
    and where for $m \leq k+j-n+1$, we have
    \[
    \beta_{m}^{k,j} = \frac{(m-1)!}{k-n+j-1} \sum_{\ell=0}^{k+j-n}  \binom{k-n+j}{\ell} c_{\ell+n+1}^{k,j},
    \]
where the constants $c_{\ell}^{k,j}$ are given in Lemma \ref{lem:Ak}.

By Lemma \ref{lem:J1J2} we can see that to complete calculating $I_{2,2}$, all we need to do is sum over $j=0, \dots, n$ and over $k \geq 1$.

In the remark following Lemma \ref{lem:Ak} we calculated $c_{0}^{k,j}$ and $c_{1}^{k,j}$ explicitly, and so we could calculate the leading and subleading terms in $I_{2,2}$, using our result for $J_1$.

In this case, to leading order we would have
    \[
    (-1)^{n+1} \frac{T}{2 \pi} L^{n+1} \sum_{k=1}^{\infty} 2^k \sum_{j=0}^{n} \binom{n}{j}\frac{c_{0}^{k,j}}{(k+j+1)!}.
    \]
Substitute the value of $c_{0}^{k,j}$ from Lemma \ref{lem:Ak} to obtain
    \begin{align*}
    &(-1)^{n+1} \frac{T}{2 \pi} \left(\log \frac{T}{2 \pi} \right)^{n+1} \sum_{k=1}^{\infty} 2^k \sum_{j=0}^{n} \binom{n}{j} \frac{(-1)^j j!}{(k+j+1)!} \\
    =& (-1)^{n+1} \frac{T}{2 \pi} \left(\log \frac{T}{2 \pi} \right)^{n+1} \sum_{k=1}^{\infty}  \left( \frac{2^k}{k!(k+n+1)} \right),
    \end{align*}
where the second line follows after summing over $j$. Now we need to sum over $k$. For this sum we obtain
    \begin{align*}
    \sum_{k=1}^{\infty}  \frac{2^k}{k! (k+n+1)} &= \frac{e^2-1}{n+1} + (-1)^{n+1} \left( \frac{\Gamma(n+2) - \Gamma (n+2,-2)}{2^{n+1} (n+1)} \right)  \\
    &= \frac{e^2-1}{n+1} + (-1)^{n+1} \frac{n!}{2^{n+1}} \left(1 - e^2 \sum_{k=0}^{n+1} \frac{(-2)^k}{k!} \right)
    \end{align*}
where we have used the incomplete gamma function in simplifying the expressions above, which for integer $n$ is given by
    \[
    \Gamma (n,x)= (n-1)!e^{-x}\sum_{k=0}^{n-1} \frac{x^k}{k!}.
    \]
This means that to leading order, we have
    \[
    I_{2,2} \sim (-1)^{n+1} \left( \frac{e^2-1}{n+1} + (-1)^{n+1} \frac{n!}{2^{n+1}} \left(1 - e^2 \sum_{k=0}^{n+1} \frac{(-2)^k}{k!} \right) \right) \frac{T}{2 \pi} \left(\log \frac{T}{2 \pi} \right)^{n+1}.
    \]

In a similar way, we can calculate the subleading behaviour, and show that this contributes
    \begin{equation*}
    (-1)^{n+1} \frac{T}{2\pi} L^{n} \left( 1-(1+e^2)\gamma_0 + (-1)^{n+1} \frac{(n+1)!}{2^{n+1}} (-1+2\gamma_0) \left( 1-e^2\sum_{k=0}^{n} \frac{(-2)^k}{k!} \right) \right).
    \end{equation*}
to the asymptotic.

Further calculations from Lemma \ref{lem:Ak} would give rise to more lower order terms in the obvious way.

We then recombine $I_{2,2}$ here with $I_{1}$ from Lemma \ref{lem:I_1} and with $I_{2,1}$ from \eqref{eq:I21}, which gives $I$.

Finally, recall that
\begin{equation*}
    \sum_{0< \lambda \leq T} \zeta^{(n)} \left( \frac{1}{2} + i\lambda \right) =  - \overline{I} + O\left(T^{1/2+\varepsilon}\right)
\end{equation*}
to obtain the asymptotic, proving Theorem \ref{thm:FirstMomnthDerivRho1}.

\section{Outline of the proof of Theorem \ref{thm:FirstMomRho1}}\label{sect:ProofCor1}
We begin by noting that the proof of Theorem \ref{thm:FirstMomRho1} follows in a similar way to that of Theorem \ref{thm:FirstMomnthDerivRho1}. We sketch the differences below but leave the full details to the interested reader.

The initial steps remain the same in the case $n=0$ as in the proof of Theorem \ref{thm:FirstMomnthDerivRho1}. We are able to write 
\begin{equation*}
    \sum_{0< \lambda \leq T} \zeta \left( \frac{1}{2} + i\lambda \right)  = - \overline{I} + O\left( T^{1/2 + \varepsilon} \right)
\end{equation*}
where
\[
I = \frac{1}{2 \pi i} \int_{c+i}^{c+iT} \frac{Z_1'}{Z_1} (1-s) \zeta (1-s) \ ds,
\]
where $c = 1+1/\log T$. 

To begin manipulating $I$ into a form that we can evaluate, we use the logarithmic derivative of the functional equation for $Z_1 (s)$, given in \eqref{eq:logderivZ1FE}, and the functional equation for $\zeta (s)$, given in \eqref{eq:FE}. Then
\begin{align*}
    I &= \frac{1}{2 \pi i} \int_{c+i}^{c+iT} \frac{Z_1'}{Z_1} (1-s) \zeta (1-s) \ ds \\
    &= - \frac{1}{2 \pi i} \int_{c+i}^{c+iT} \log \frac{t}{2 \pi} \chi (1-s) \zeta(s) \ ds - \frac{1}{2 \pi i} \int_{c+i}^{c+iT}  \frac{Z_1'}{Z_1} (s) \chi (1-s) \zeta (s) \ ds + O\left(T^{1/2+\varepsilon} \right)\\\
    &=I_1 + I_2 + O\left(T^{1/2+\varepsilon} \right),
\end{align*}
say. 

Then by Lemma \ref{lem:I_1}, we have
\begin{equation}\label{eq:I_1Zeta}
    I_1 = -\frac{T}{2\pi} \log \frac{T}{2\pi} + \frac{T}{2\pi} + O \left( T^{1/2 + \varepsilon}\right).
\end{equation}

For $I_2$, we have
\begin{align*}
    I_2 &= - \frac{1}{2 \pi i} \int_{c+i}^{c+iT}  \frac{Z_1'}{Z_1} (s) \chi (1-s) \zeta (s) \ ds \\
    &= \frac{1}{2 \pi i} \int_{c+i}^{c+iT} \chi (1-s) \sum_{n=1}^{\infty} \frac{\Lambda (n)}{n^s} \sum_{m=1}^{\infty} \frac{1}{m^s} \ ds \\
    &\qquad - \sum_{k=1}^{\infty} \frac{1}{2 \pi i} \int_{c+i}^{c+iT} \chi (1-s) \frac{1}{f(s)^k} \sum_{n=1}^{\infty}  \frac{a_k (n)}{n^s} \sum_{m=1}^{\infty} \frac{1}{m^s} \ ds\\
    &= I_{2,1} + I_{2,2},
\end{align*}
say.

Then $I_{2,1}$ can be evaluated through a simple application of Lemma \ref{lem:StatPhase} and Perron's theorem, giving
\begin{equation}\label{I_2,1Zeta}
    I_{2,1} = \frac{T}{2\pi} \log \frac{T}{2\pi} - \frac{T}{2\pi} + O \left( T^{1/2 + \varepsilon}\right).
\end{equation}
Combining $I_1$  from \eqref{eq:I_1Zeta} and $I_{2,1}$ from \eqref{I_2,1Zeta} gives a contribution of $O \left( T^{1/2 + \varepsilon}\right)$ to the asymptotic.

Finally, to evaluate $I_{2,2}$, we have after using Lemma \ref{lem:StatPhase},
\begin{align*}
    I_{2,2} &= -\sum_{k=1}^{\infty} \frac{1}{2 \pi i} \int_{c+i}^{c+iT} \chi (1-s) \frac{1}{f(s)^k}\sum_{n=1}^{\infty} \frac{a_k (n)}{n^s} \sum_{m=1}^{\infty} \frac{1}{m^s} \ ds\\
    &= -\sum_{k=1}^{\infty} 2^k \sum_{nm \leq T/2\pi} \frac{a_k(n)}{(\log nm)^k} + O\left( T^{1/2 +\varepsilon} \right).
\end{align*} 

As in the proof of Theorem \ref{thm:FirstMomnthDerivRho1}, we evaluate the numerator of the inner sum in the previous line via Perron. This is actually a special case of Lemma \ref{lem:Ak} in the case $j=0$, so we have the following result.

\begin{lemma}\label{lem:Ak2}
    Let 
    \[
    A_{k}(x) = \sum_{mn \leq x} a_{k}(n),
    \]
    where $a_k(n)$ is given in \eqref{eq:ak(n)}. This sum can also be written as
    \[
    A_{k}(x) = \sum_{1 \leq m n_1 n_2 ... n_k  \leq x} \log (n_1) \Lambda (n_1) \Lambda (n_2) \dots \Lambda (n_k) .
    \]
    Then for large $x$,
    \[
    A_{k}(x) = x \sum_{\ell=0}^{k+1} \frac{c_{k,\ell}}{(k+1-\ell)!} (\log x)^{k+1-\ell} + O\left(x^{1/2+\varepsilon}\right)
    \]
    where the $c_{k,\ell}$ are the Laurent series coefficients around $s=1$ of
    \[
    \left(\frac{\zeta '}{\zeta} (s) \right)' \left(-\frac{\zeta '}{\zeta} (s) \right)^{k-1} \zeta (s) \frac{1}{s} = \sum_{\ell=0}^\infty c_{k,\ell} (s-1)^{-k-2+\ell} .
    \]
\end{lemma}
    
\begin{remark}
    A long but straightforward calculation tells us what the coefficients $c_{k,\ell}$ equal. The leading coefficient is when $\ell = 0$ and is given by
    \[
    c_{k,0} = 1.
    \]
    When $\ell = 1$, the subleading coefficient is given by
    \[
    c_{k,1} = -1 + \gamma_0 + (1-k) \gamma_0.
    \]
\end{remark}

Next we use partial summation to calculate 
\[
\sum_{nm \leq T/2\pi} \frac{a_k(n)}{(\log nm)^k} = A_k \left( \frac{T}{2 \pi} \right) f\left( \frac{T}{2 \pi} \right) - \int_{2}^{T/2\pi} A_k(x) f'(x) \ dx,
\]
in the case $f(x) = 1/(\log x)^k$ (so $f'(x) = -k/(x(\log x)^{k+1})$). Then using $A_k(x)$ from Lemma \ref{lem:Ak2}, we have with $L = \log T/2\pi$, 
\begin{multline*}
    \sum_{nm \leq T/2\pi} \frac{a_k(n)}{(\log nm)^k} = \\
    \frac{T}{2 \pi} \sum_{\ell=0}^{k+1} \frac{c_{k,\ell}}{(k+1-\ell)!} L^{1-\ell} 
    + k \sum_{\ell=0}^{k+1} \frac{c_{k,\ell}}{(k+1-\ell)!} \int_{2}^{T/2\pi} (\log x)^{-\ell} \ dx + O\left( T^{1/2 +\varepsilon} \right).
\end{multline*}

As before, we extract the non-negative powers of the logarithm and calculate the coefficients of the negative powers as we clearly have an infinite chain of decreasing powers of the logarithm. This gives
\begin{multline*}
    \sum_{nm \leq T/2\pi} \frac{a_k(n)}{(\log nm)^k} = \frac{T}{2 \pi} \frac{c_{k,0}}{(k+1)!} L + \frac{T}{2 \pi} \frac{c_{k,1}}{k!} + \frac{T}{2 \pi} k \frac{c_{k,0}}{(k+1)!} \\
    + \frac{T}{2 \pi} \sum_{\ell=2}^{k+1} \frac{c_{k,\ell}}{(k+1-\ell)!} L^{1-\ell} + k \sum_{\ell=1}^{k+1} \frac{c_{k,\ell}}{(k+1-\ell)!} \int_{2}^{T/2\pi} (\log x)^{-\ell} \ dx + O\left( T^{1/2 +\varepsilon} \right).
\end{multline*}

For the negative powers of the logarithm, we reindex so both sums start at $\ell=1$ to give
\[
\frac{T}{2 \pi} \sum_{\ell=1}^{k} \frac{c_{k,\ell+1}}{(k-\ell)!} L^{-\ell} + k \sum_{\ell=1}^{k+1} \frac{c_{k,\ell}}{(k+1-\ell)!} \int_{2}^{T/2\pi} (\log x)^{-\ell} \ dx.
\]

Recall that for $1 \leq m \leq M$, we have
\begin{align*}
    \int_{2}^{T/2\pi} \frac{1}{(\log x)^m} \ dx & = \frac{T}{2\pi}\frac{1}{L^m} + m \int_{2}^{T/2\pi} \frac{1}{(\log x)^{m+1}} \ dt + O(1)\\
    &=  \frac{T}{2\pi} \sum_{r=m}^{M} \frac{(r-1)!}{(m-1)!}\frac{1}{L^{r}} 
   + O\left( \frac{T}{L^{M+1}} \right).
\end{align*}

We now write $\beta_m^{k,\ell}$ for the coefficient of $L^{-m}$. In an analogous way to to the proof of Theorem \ref{thm:FirstMomnthDerivRho1}, we have for $m \leq k$, 
    \[
    \beta_{m}^{k,j} = \frac{c_{k,m+1}}{(k-m)!} + k \frac{(m-1)!}{(k-1)!} \sum_{\ell=0}^{m-1} \binom{k}{\ell} c_{k,\ell+1} 
    \] 
and where for $m \leq k+j-n+1$, we have
    \[
    \beta_{m}^{k,j} =  k \frac{(m-1)!}{(k-1)!} \sum_{\ell=0}^{k} \binom{k}{\ell} c_{k,\ell+1} ,
    \]
where the constants $c_{k,\ell}$ are given in Lemma \ref{lem:Ak2}.

In order to complete the proof, we multiply through by $-2^k$, and sum over all $k \geq 1$ to obtain $I_{2,2}$. Since the combination from the rest of the calculation only contributes to the error term, we have evaluated $I$ and so the sum in question. 

Note that since we have a simple form for $c_{k,0}$ and $c_{k,1}$, we can easily calculate the leading and subleading behaviour explicitly, giving for $K$ a positive integer
\begin{equation*}
    \sum_{0< \lambda \leq T} \zeta \left( \frac{1}{2} + i\lambda \right)  = \frac{e^2-3}{2} \frac{T}{2 \pi} L + \frac{3-e^2-4\gamma_0}{2} \frac{T}{2 \pi} + \frac{T}{2 \pi} \sum_{m=1}^{K} \frac{d_m}{L^m} + O_K\left(\frac{T}{L^{K+1}}\right)
\end{equation*}
as $T \rightarrow \infty$, where 
\[
d_m = \sum_{k=1}^{\infty} 2^k \beta_{m}^{k,\ell},
\]
with the $\beta_{m}^{k,\ell}$ as above.

\section{Proof Theorem \ref{thm:ChiRho1}}\label{sect:ProofThmChi1}
The main steps in the proof of Theorem \ref{thm:ChiRho1} follow in a similar way to Theorem \ref{thm:FirstMomnthDerivRho1}. We sketch the steps here for completeness sake.

Under the Riemann Hypothesis we can show that all the zeros off the critical line only contribute to the error term in our asymptotic with a similar justification to the proof of Theorem \ref{thm:FirstMomnthDerivRho1}. Recall $\chi (s) \ll T^{1/2-\sigma}$ for $s=\sigma+it$ with $|t|\ll T$. Write $\rho_{1} = \beta_{1} + i \gamma_{1}$ for a zero of $Z_1(s)$. By part (3) of Lemma \ref{lem:CGLemma}, $Z_1(s)$ has $O(\log T)$ zeros off the critical line (that is, $\beta_1 \neq 1/2$). At the zeros off the critical line, we may use part (4) of Lemma \ref{lem:CGLemma} to obtain
\[
\sum_{\substack{0< \gamma_{1} \leq T \\ \beta_1 \neq 1/2}} \chi (\rho_{1}) \ll T^{1/9} \log T \ll T^{1/9 + \varepsilon}.
\]

Therefore, we have
\begin{equation*}
    \sum_{0< \lambda \leq T} \chi \left( \frac{1}{2} + i\lambda \right) =  \sum_{\substack{0 < \gamma_1 \leq T \\ \beta_1 = 1/2}} \chi \left( \frac{1}{2} + i\gamma_1 \right).
\end{equation*}

We can write this as an integral using Cauchy's theorem, 
\begin{align}\label{eq:CauchyChi}
\frac{1}{2\pi i} \int_{\mathcal{C}} \frac{Z_1'}{Z_1}(s) \chi (s) \ ds &= \sum_{0 < \gamma_1 \leq T} \chi (\rho_1) \notag  \\
&= \sum_{\substack{0 < \gamma_1 \leq T \\ \beta_1 = 1/2}} \chi \left( \frac{1}{2} + i\gamma_1 \right) + O\left(T^{1/9 + \varepsilon} \right)
\end{align}
where $\mathcal{C}$ is a positively oriented contour with vertices $c + i$, $c + iT$, $1-c+iT$, and $1-c +i$, where $c=1+1/\log T$. We may assume, without loss of generality, that the distance from the contour to any zero $\rho_1$ of $Z_1(s)$ is uniformly $\gg 1/\log T$.

We split the integral as 
\begin{align*}
&\frac{1}{2 \pi i} \left( \int_{c+i}^{c+iT} + \int_{c+iT}^{1-c+iT} + \int_{1-c+iT}^{1-c+i} + \int_{1-c+i}^{c+i} \right) \frac{Z_1 '}{Z_1}(s) \chi (s) \ ds \\
&= S^R + S^T + S^L + S^B,
\end{align*}
say. Note that $S^B, S^T, S^R$ are all trivially bounded within error term in a similar way to the proof of Theorem \ref{thm:FirstMomnthDerivRho1}, using the bound $\chi (s) \ll T^{1/2-\sigma}$ where appropriate. This shows that the leading contributions to the asymptotic will then come from $S^L$.

Since $S^B, S^T, S^R$ are all within an error term of $O\left(T^{1/2+\varepsilon}\right)$, all that remains is to evaluate $S^L$. Note that 
\[
S^L = \frac{1}{2 \pi i} \int_{1-c+iT}^{1-c+i} \frac{Z_1 '}{Z_1}(s) \chi (s) \ ds = -\frac{1}{2 \pi i} \int_{c-iT}^{c-i} \frac{Z_1 '}{Z_1}(1-s) \chi (1-s) \ ds = - \overline{I},
\]
where
\begin{equation}
I = \frac{1}{2 \pi i} \int_{c+i}^{c+iT} \frac{Z_1 '}{Z_1}(1-s) \chi (1-s) \ ds. \label{eq:IChi}
\end{equation}

Overall, we have 
\begin{equation*}
    \sum_{0< \lambda \leq T} \chi \left( \frac{1}{2} + i\lambda \right) =  - \overline{I} + O\left(T^{1/2+\varepsilon}\right).
\end{equation*}

Using the version of the functional equation for $Z_1'(s)/Z_1 (s)$ given in \eqref{eq:logderivZ1FE2}, we have
\begin{multline*}
    I = \frac{1}{2 \pi i} \int_{c+i}^{c+iT} \left( - \log \frac{t}{2 \pi} \right) \chi (1-s) \ ds \\
    + \frac{1}{2 \pi i} \int_{c+i}^{c+iT} \left( - \frac{Z_1'}{Z_1} (s) \right) \chi (1-s) \ ds + O\left(T^{1/2+\varepsilon}\right) \\
    = I_{1} + I_{2} + O\left(T^{1/2+\varepsilon}\right),
\end{multline*}
say.

Then $I_1$ doesn't contribute to anything beyond an error term, so we split $I_2$ as before to write
\begin{multline*}
    I_2 = \frac{1}{2 \pi i} \int_{c+i}^{c+iT} \left( \sum_{n=1}^{\infty} \frac{\Lambda (n)}{n^s} \right) \chi (1-s) \ ds \\
    - \sum_{k=1}^{\infty} \frac{1}{2 \pi i} \int_{c+i}^{c+iT} \chi (1-s) \frac{1}{f(s)^k} \sum_{n=1}^{\infty} \frac{a_k (n)}{n^s} \ ds + O\left(T^{1/2+\varepsilon}\right)\\
    = I_{2,1} + I_{2,2} + O\left(T^{1/2+\varepsilon}\right).
\end{multline*}

Then applying Lemma \ref{lem:StatPhase} to $I_{2,1}$ we have
\[
I_{2,1} = \sum_{n \leq T/2\pi} \Lambda (n) + O\left(T^{1/2+\varepsilon}\right).
\]
By Perron/the Prime Number Theorem, we have
\begin{equation}\label{eq:ChiI21}
I_{2,1} = \frac{T}{2 \pi} + O\left(T^{1/2+\varepsilon}\right).
\end{equation}

Finally, applying Lemma \ref{lem:StatPhase} to $I_{2,2}$ gives
\[
I_{2,2} = - \sum_{k=1}^{\infty} 2^k \sum_{n \leq T/2\pi} \frac{a_k(n)}{(\log n)^k} + O\left(T^{1/2+\varepsilon}\right).
\]

As in the proof of Theorem \ref{thm:FirstMomnthDerivRho1}, we evaluate the numerator of the inner sum in the previous line via Perron. 

\begin{lemma}\label{lem:Ak3}
    Let 
    \[
    A_{k}(x) = \sum_{n \leq x} a_{k}(n),
    \]
    where $a_k(n)$ is given in \eqref{eq:ak(n)}. This sum can also be written as
    \[
    A_{k}(x) = \sum_{1 \leq  n_1 n_2 ... n_k  \leq x} \log (n_1) \Lambda (n_1) \Lambda (n_2) \dots \Lambda (n_k) .
    \]
    Then for large $x$,
    \[
    A_{k}(x) = x \sum_{\ell=0}^{k} \frac{c_{k,\ell}}{(k-\ell)!} (\log x)^{k-\ell} + O\left(x^{1/2+\varepsilon}\right)
    \]
    where the $c_{k,\ell}$ are the Laurent series coefficients around $s=1$ of
    \[
    \left(\frac{\zeta '}{\zeta} (s) \right)' \left(-\frac{\zeta '}{\zeta} (s) \right)^{k-1}  \frac{1}{s} = \sum_{\ell=0}^\infty c_{k,\ell} (s-1)^{-k-1+\ell} .
    \]
\end{lemma}
    
\begin{remark}
    A long but straightforward calculation tells us what the coefficients $c_{k,\ell}$ equal. The leading coefficient is when $\ell = 0$ and is given by
    \[
    c_{k,0} = 1.
    \]
    When $\ell = 1$, the subleading coefficient is given by
    \[
    c_{k,1} = -1 + \gamma_0 -k \gamma_0.
    \]
\end{remark}

Next we use partial summation to calculate 
\[
\sum_{n \leq T/2\pi} \frac{a_k(n)}{(\log n)^k} = A_k \left( \frac{T}{2 \pi} \right) f\left( \frac{T}{2 \pi} \right) - \int_{2}^{T/2\pi} A_k(x) f'(x) \ dx,
\]
in the case $f(x) = 1/(\log x)^k$ (so $f'(x) = -k/(x(\log x)^{k+1})$). Then using $A_k(x)$ from Lemma \ref{lem:Ak3}, we have with $L = \log T/2\pi$, 
\begin{multline*}
    \sum_{n \leq T/2\pi} \frac{a_k(n)}{(\log n)^k} = \\
    \frac{T}{2 \pi} \sum_{\ell=0}^{k} \frac{c_{k,\ell}}{(k-\ell)!} L^{-\ell} 
    + k \sum_{\ell=0}^{k} \frac{c_{k,\ell}}{(k-\ell)!} \int_{2}^{T/2\pi} (\log x)^{-1-\ell} \ dx + O\left( T^{1/2 +\varepsilon} \right).
\end{multline*}

As before, we extract the non-negative powers of the logarithm and calculate the coefficients of the negative powers as we clearly have an infinite chain of decreasing powers of the logarithm. This gives
\begin{multline*}
    \sum_{n \leq T/2\pi} \frac{a_k(n)}{(\log n)^k} = \frac{T}{2 \pi} \frac{c_{k,0}}{k!} L \\
    + \frac{T}{2 \pi} \sum_{\ell=1}^{k} \frac{c_{k,\ell}}{(k-\ell)!} L^{-\ell} + k \sum_{\ell=0}^{k} \frac{c_{k,\ell}}{(k-\ell)!} \int_{2}^{T/2\pi} (\log x)^{-1-\ell} \ dx + O\left( T^{1/2 +\varepsilon} \right).
\end{multline*}

For the negative powers of the logarithm, we reindex so both sums start at $\ell=1$ to give
\[
\frac{T}{2 \pi} \sum_{\ell=1}^{k} \frac{c_{k,\ell}}{(k-\ell)!} L^{-\ell} + k \sum_{\ell=1}^{k+1} \frac{c_{k,\ell-1}}{(k+1-\ell)!} \int_{2}^{T/2\pi} (\log x)^{-\ell} \ dx.
\]

Recall that for $1 \leq m \leq M$, we have
\begin{align*}
    \int_{2}^{T/2\pi} \frac{1}{(\log x)^m} \ dx & = \frac{T}{2\pi}\frac{1}{L^m} + m \int_{2}^{T/2\pi} \frac{1}{(\log x)^{m+1}} \ dt + O(1)\\
    &=  \frac{T}{2\pi} \sum_{r=m}^{M} \frac{(r-1)!}{(m-1)!}\frac{1}{L^{r}} 
   + O\left( \frac{T}{L^{M+1}} \right).
\end{align*}

We now write $\beta_m^{k,\ell}$ for the coefficient of $L^{-m}$. In an analogous way to to the proof of Theorem \ref{thm:FirstMomnthDerivRho1}, we have for $m \leq k$, 
    \[
    \beta_{m}^{k,j} = \frac{c_{k,m}}{(k-m)!} + \frac{(m-1)!}{(k-1)!} \sum_{\ell=0}^{m-1} \binom{k}{\ell} c_{k,\ell} 
    \] 
and where for $m \leq k+j-n+1$, we have
    \[
    \beta_{m}^{k,j} =  \frac{(m-1)!}{(k-1)!} \sum_{\ell=0}^{k} \binom{k}{\ell} c_{k,\ell} ,
    \]
where the constants $c_{k,\ell}$ are given in Lemma \ref{lem:Ak3}.

In order to complete the proof, we multiply through by $-2^k$, and sum over all $k \geq 1$ to obtain $I_{2,2}$. Once we add in the contribution from $I_{2,1}$ given in \eqref{eq:ChiI21}, we have evaluated $I$ and so the sum in question. 

Note that since we have a simple form for $c_{k,0}$ and $c_{k,1}$, we can easily calculate the leading and subleading behaviour explicitly, giving for $K$ a positive integer
\begin{equation*}
    \sum_{0< \lambda \leq T} \chi \left( \frac{1}{2} + i\lambda \right)  = (e^2-2) \frac{T}{2 \pi} - 4e^2\gamma_0 \frac{T}{2 \pi} \frac{1}{L} + \frac{T}{2 \pi} \sum_{m=2}^{K} \frac{e_m}{L^m} + O_K\left(\frac{T}{L^{K+1}}\right)
\end{equation*}
as $T \rightarrow \infty$, where 
\[
e_m = \sum_{k=1}^{\infty} 2^k \beta_{m}^{k,\ell},
\]
with the $\beta_{m}^{k,\ell}$ as above.

\section{Outline of the proof of Theorem \ref{thm:ChiRho}}\label{sect:ProofThmChi}
The proof of Theorem \ref{thm:ChiRho} is very similar to that of Theorem \ref{thm:ChiRho1}. We just sketch out the main steps here. Note that this result can be made unconditional but we assume the Riemann Hypothesis to make the error terms trivial to deal with.

We begin by using Cauchy's theorem to write 
\[
\sum_{0< \gamma \leq T} \chi \left( \frac{1}{2} + i\gamma \right) = \frac{1}{2\pi i} \int_{\mathcal{C}} \frac{\zeta'}{\zeta}(s) \chi (s) \ ds, 
\]
where $\mathcal{C}$ is a positively oriented contour with vertices $c + i$, $c + iT$, $1-c+iT$, and $1-c +i$, where $c=1+1/\log T$. We may assume, without loss of generality, that the distance from the contour to any zero of $\zeta (s)$ is uniformly $\gg 1/\log T$.

Note that the integrals along the bottom, top and right-hand side of the contour are bounded by $O\left( T^{1/2+\varepsilon}\right)$, using the bound $\chi (s) \ll T^{1/2-\sigma}$ for $s=\sigma+it$ with $|t|\ll T$. As before, the only contribution comes from the left-hand side of the contour, and as before, we have 
\[
\frac{1}{2 \pi i} \int_{1-c+iT}^{1-c+i} \frac{\zeta}{\zeta}(s) \chi (s) \ ds = -\frac{1}{2 \pi i} \int_{c-iT}^{c-i} \frac{\zeta}{\zeta}(1-s) \chi (1-s) \ ds = - \overline{I},
\]
where
\begin{equation*}
I = \frac{1}{2 \pi i} \int_{c+i}^{c+iT} \frac{\zeta}{\zeta}(1-s) \chi (1-s) \ ds.
\end{equation*}

Overall, we have 
\begin{equation*}
    \sum_{0< \gamma \leq T} \chi \left( \frac{1}{2} + i\gamma \right) =  - \overline{I} + O\left(T^{1/2+\varepsilon}\right).
\end{equation*}

Next we need the functional equation for $\zeta'(s)/\zeta (s)$, which can be derived easily from the functional equation for $\zeta (s)$ in \eqref{eq:FE}. This gives
\begin{equation*}
\frac{\zeta '}{\zeta} (1-s) = -\log \frac{t}{2\pi} - \frac{\zeta '}{\zeta} (s) + O\left( \frac{1}{|t|} \right).
\end{equation*}
Substituting this into $I$ gives
\begin{multline*}
I = - \frac{1}{2 \pi i} \int_{c+i}^{c+iT} \log \frac{t}{2\pi} \chi (1-s) \ ds \\ 
- \frac{1}{2 \pi i} \int_{c+i}^{c+iT} \frac{\zeta'}{\zeta} (s) \chi (1-s) \ ds + O\left(T^{1/2+\varepsilon}\right) \\
= I_1 + I_2 + O\left(T^{1/2+\varepsilon}\right).
\end{multline*}
Clearly $I_1$ doesn't contribute to the asymptotic, so consider $I_2$. Then by Lemma \ref{lem:StatPhase},
\[
I_2 = \sum_{n \leq T/2\pi} \Lambda (n) + O\left(T^{1/2+\varepsilon}\right).
\]
Then by Perron/the Prime Number Theorem,
\[
I_2 = \frac{T}{2\pi} + O\left(T^{1/2+\varepsilon}\right).
\]
The result then follows by working back up the proof.

\begin{remark}
    We could see this result directly from the proof of Theorem \ref{thm:ChiRho1}, noticing that the logarithmic derivative of $\zeta (s)$ is included in that of $Z_1 (s)$, and making the obvious changes to the proof. We have included the basic steps above for clarity. 
\end{remark}

\section*{Acknowledgements}
The majority of this paper constitutes the last chapter of my PhD thesis \cite{theAPC}. I would like to thank Chris Hughes and Sol Lugmayer for the many helpful discussions we had while we wrote \cite{HugLugPea24}, which prompted me to consider the problems found in this paper. Thanks also goes to Chris Hughes for generating the graphs found in this paper. Thanks also to Nathan Ng for some helpful suggestions on how to write the results.

\bibliographystyle{abbrv}
\bibliography{bibliography}	

\end{document}